\documentclass[12pt,twoside]{article} 

\usepackage{orcidlink}
\usepackage[utf8]{inputenc}
\usepackage{amsmath}
\usepackage{amssymb}
\usepackage{amsthm}
\usepackage{mathrsfs}
\usepackage{dsfont}
\usepackage{booktabs} 
\usepackage{caption} 
\usepackage{subcaption} 
\usepackage{tikz}
\usetikzlibrary{er,positioning,bayesnet}
\usepackage{multicol}
\usepackage{hyperref}
\usepackage[inline]{enumitem} 
\usepackage{csquotes}
\usepackage{array}
\usepackage[verbose]{placeins}
\usepackage{comment}
\usepackage{xcolor}
\usepackage{mathtools}
\usepackage{titlesec}
\usepackage{makecell}

\titleformat{\subsection}
  {\normalfont\bfseries\itshape}  
  {\thesubsection}       
  {1em}                  
  {}

\usepackage{natbib}

\newtheorem{satz}{Theorem}[section]
\newtheorem{lemma}[satz]{Lemma}

\newtheorem{remark}[satz]{Remark}

\newtheorem{bsp}[satz]{Example}
\newtheorem{ann}{Assumption}

\DeclareMathOperator*{\argmin}{arg\,min}

\begin{document}

\title{\bf Local polynomial estimation of quantile density functions
}

\author{Niclas Jacobsen and Natalie Neumeyer\footnote{e-mail: natalie.neumeyer@uni-hamburg.de; ORCID: \orcidlinkf{0000-0002-6649-1135}}\\ University of Hamburg, Department of Mathematics}

\maketitle

\begin{abstract}
A new approach for nonparametric estimation of the quantile density function (sparsity function) and its derivatives is suggested which is based on local polynomial estimation. The estimator has more advantageous properties at the boundaries than classical quantile density estimators. Asymptotic normality is shown and the bias, asymptotic variance as well as boundary properties are compared with other estimators.  
\end{abstract}




\section{Introduction}

Let \( X \) be a real-valued absolutely continuous random variable with cumulative distribution function \( F \) and density function \( f \). The quantile function of \( X \) is defined as $Q(u)=F^{-1}(u)= \inf\{x\in \mathbb{R}:   F(x)\geq u\}$ for $u\in (0,1)$, and the quantile density function is the derivative
$$ q(u)=Q'(u)= \frac{1}{f(Q(u))}, \qquad u\in (0,1).$$
Estimation of the quantile density (also called sparsity function, see \cite{Tukey})  is important because it appears in the expression for the asymptotic variance of empirical and kernel type estimators of the quantile function, see e.g.\ \cite{Serfling}, p.\ 77, and \cite{Sheather-Marron}.  It is also important in survival analysis because the hazard quantile function can be written in terms of the quantile density function, see  \cite{Nair-Sankaran} or \cite{Sankaran-Nair}.  
Recently as an application in finance the investigation of price auctions via estimation of the quantile density function of bids has been considered by \cite{Zhang}, \cite{Shakhgildyan} and \cite{Doosti-etal}.
Also the quantile function appears in the definition of ROC-curves, and the quantile density function is the slope of the ROC-curve, which is an important feature in evaluating
diagnostic tests, see e.g.\ \cite{Pepe}. 
Further \cite{Redivo-etal} consider classification methods based on quantile density functions. 
\cite{Petersen-Mueller} consider log quantile density transformation. 
Early considerations of estimating densities evaluated in a quantile are by \cite{Bloch-Gastwirth}, \cite{Bofinger}, and \cite{Reiss}, among others. 
Estimators for the quantile density function based on kernel estimation methods  have been considered by \cite{parzen1979nonparametric}, \cite{falk1986estimation}, \cite{welsh}, \cite{Sheather-Marron}, \cite{jones1992estimating}, and \cite{soni2012nonparametric}. \cite{Zhou-Yip} considered kernel estimators for quantile functions based on truncated and censored data. Wavelet estimation methods have been considered by \cite{chesneau2016nonparametric}, and Poisson-based as well as Bernstein polynomial estimators by \cite{Chaubey-etal}. 
\cite{soni2012nonparametric} show that the kernel-based estimators do not perform well at the boundaries. 
\\
We consider a new estimator for the quantile density based on a local polynomial estimation procedure which is boundary adaptive. The local polynomial procedure  is a classic method  to estimate nonparametric regression functions and their derivatives, see \cite{Fan-Gijbels} for an extensive study, which extended methods by \cite{Stone1977}, \cite{Cleveland}, among others.  Local polynomial estimation has very good asymptotic properties, in particular concerning the covariate support boundary in comparison to other kernel estimation methods. The method can also be applied to estimate quantile regression functions and expectile regression functions, see for instance  \cite{ElGhouch-Genton}, and \cite{Adam-Gijbels}. Local polynomial estimation is still widely used and modified in current research, see for instance \cite{cattaneo2020simple}, \cite{Bouanani-Bouzebda}, \cite{Cattaneo-etal-2024},  \cite{Bertin-etal}, and \cite{Jirak-etal}.
\\
Consider an independent sample $X_1,\dots,X_n$ with the same distribution as $X$.
Our new idea to estimate the quantile density $Q'=q$ is based on the formula $$Q(u)=E[X_i\mid Z_i=u] \mbox{ with }Z_i=F(X_i).$$ Then applying the local polynomial regression estimator on the data  $( Z_i, X_i)$, $i=1,\dots,n$, by interpreting $ Z_i$ as covariates and $X_i$ as regression responses, would lead to an estimator for $Q(u)$. But those random variables $Z_i$ are unknown and need to be replaced by pseudo-data $\hat Z_i=F_n(X_i)$ with the empirical cdf $F_n$ of $X_1,\dots,X_n$.
Note that \cite{cattaneo2020simple}  estimated the density $F^\prime(x)=f(x)$ and derivatives based on the formula $F(x)=E[Z_i\mid X_i=x]$. Reverse we obtain local polynomial  estimators for the quantile function $Q$, the quantile density function $q$, and also for derivatives of $q$. Estimators for derivatives are also important, for example to estimate the score function $q'/q^2$ considered by \cite{parzen1979nonparametric} which is related to hazard rates   {and applied in the article \cite{Nair-etal} to classify lifetime models. The quantile function derivatives are also of interest to investigate the structure of the data generating distribution. Inflection points of the quantile function are those with $Q^{(2)}(u)=0$, which correspond to local maxima or minima (modes) of the density function. Higher derivatives of the quantile function also appear in \cite{Staudte} in order to characterize tail behavior of distributions. Further \cite{Prendergast-Staudte} introduce the quantile optimality ratio $q(u)/q''(u)$ in the context of choosing optimal bandwidths for quantile density estimators when constructing confidence intervals for quantile estimators.
The new estimator has the same simple structure
as local polynomial regression estimators}, and has much better properties at the boundaries than the classical kernel-based quantile density estimators in both cases of bounded and unbounded support of the distribution of $X$. Investigating the asymptotic properties is more challenging than for classical local polynomial procedures based on iid data because the $\hat Z_i$ are dependent pseudo-observations based on the empirical distribution function, or rewriting the estimator it is based on dependent order statistics.
\\
In section 2 we will define the estimators for the derivatives $Q^{(v)}$, $v=0,\dots,p$, based on the local polynomial of order $p$, and show the asymptotic normality result. In section 3 we consider in particular  estimating $q=Q^{(1)}$. We compare bias, asymptotic variance and boundary behaviour with different quantile density estimators from the literature. Further we discuss the case of unbounded support of the distribution of $X$. In section 4 we show simulation results. Section 5 concludes the paper, and proofs are presented in the appendix. 

\section{Definition of the Estimators and Main Result}
\label{Section: Definition of the estimators and main result}

 {As motivated in the introduction we write the quantile function as $Q(u)=E[X_i\mid F(X_i)=u]$, where $F$ is the cdf of $X_i$. Then we define the local polynomial estimator based on pseudo-data $({F}_n(X_i),X_i)$, where $ F_n$ is the empirical distribution function of $X_1,\dots,X_n$.}
 We consider local polynomial estimation of order $p\in\mathbb{N}$.  Further, let the kernel $K$ be a symmetric density with support $[-1,1]$, and the bandwidth $h=h_n$ a positive sequence with $h\to0$, $nh\to\infty$.  {We use the notation $A^\top$ for the transposed version of a matrix or vector $A$}, and further use the notations  $r_p(u) = (1,u,\dots,u^p)^\top $ and $x\wedge y=\min(x,y)$. We define 
\begin{eqnarray}
    \label{Gleichung: Definition von beta über arg min mit X_i}
    \hat{\beta}_p(u) &=& \argmin_{b\in\mathbb{R}^{p+1}} \sum_{i=1}^n \left(X_i - r_p(F_n(X_i) - u)^\top b\right)^2K\left(\frac{F_n(X_i)-u}{h}\right) \nonumber\\
       \label{Gleichung: Definition von beta über arg min mit i/n}
   &=& \argmin_{b\in\mathbb{R}^{p+1}} \sum_{i=1}^n \left(X_{(i)} - r_p\left(\frac{i}{n} - u\right)^\top b\right)^2K\left(\frac{\frac{i}{n}-u}{h}\right)
\end{eqnarray}
for $u\in [0,1]$, and where $X_{(1)}<\dots< X_{(n)}$ denote the order statistics of $X_1,\dots,X_n$ (strictly increasing with probability one). Note that the estimator has the form of a local polynomial regression estimator for equidistant design points $i/n$ and responses $X_{(i)}$. Classical results for asymptotics of local polynomial estimators cannot be directly applied due to the dependence of the order statistics. 
 For estimating the $v$-th derivative of $Q$ for $v\in\{0,\dots,p\}$ we use the notation $e_v$ for the $(p+1)$-dimensional standard vector with value 1 in the component $v+1$, and all other components are zero. Then the $v$-th derivative $Q^{(v)}(u)$ is estimated by 
\begin{eqnarray}
\label{Einführung: Gleichung: Defintion Schätzer}
   \hat{Q}^{(v)}(u)= \hat{Q}^{(v)}_p(u) = e^\top _vv!\hat{\beta}_p(u).
\end{eqnarray} 
We use the notation $\hat{Q}^{(v)}_p$ when we compare the estimators for different orders $p$ later, but otherwise we drop the $p$ in the notation. We are particularly interested in the estimator $\hat q=\hat q_p=\hat{Q}_p^{(1)}$ of the quantile density, but also higher derivatives are of interest, and we derive the asymptotic results for the general case. 

\begin{remark}\label{matrizendarstellung}\rm It holds that
\begin{eqnarray*}
           \hat \beta _p(u)  &=& \,(D_{ {p,u}}^\top K_{h,u}D_{ {p,u}})^{-1}D_{ {p,u}}^\top K_{h,u}X_{(\cdot)}
\end{eqnarray*}
with the matrix $D_{ {p,u}} = [(i/n-u)^j)]_{1\leq i\leq n,0\leq j\leq p}$, the diagonal matrix $K_{h,u}$ that has entries $(K_h(i/n-u))_{1\leq i\leq n}$ using the kernel notation $K_h(\cdot) = K(\cdot/h)/h$, and with $X_{(\cdot)}$ the ordered vector $(X_{(1)},\dots,X_{(n)})^\top $.  {All objects defined here also depend on $n$ but we choose not to add an additional $n$ in every subscript for easier readability. Note, that using this expression it is obvious that $\hat Q^{(v)}_p(u) = e_v^\top v!\hat \beta_p(u)$ is a L-statistic.} 
\end{remark}

\begin{remark}\rm  For the quantile density estimator with $v=1,p=1$ the formula
$$\hat q_1(u)=\frac{s_{n,0}(u)s_{n,1,X}(u)-s_{n,1}(u)s_{n,0, X}(u)}{s_{n,0}(u)s_{n,2}(u)-s_{n,1}(u)^2}$$
holds (analogous to section 2.2 in \cite{Fan-Gijbels}) with $s_{n,j}(u) = \sum_{i=1}^n K_h(i/n-u)(i/n-u)^j$ and $s_{n,j,X}(u) = \sum_{i=1}^n K_h(i/n-u)(i/n-u)^jX_{(i)}$. 
Note that $s_{n,0}(u)s_{n,2}(u)-s_{n,1}(u)^2>0$ by Cauchy Schwarz inequality. Further one can rewrite
\begin{eqnarray*}
  &&  s_{n,0}(u)s_{n,1,X}(u)-s_{n,1}(u)s_{n,0,X}(u)\\
  &=& 
   \sum_{j=1}^n\sum_{i=1}^{j-1} K_h(i/n-u)K_h(j/n-u)\frac{j-i}{n}(X_{(j)}-X_{(i)})\; \geq 0,
\end{eqnarray*}
and one obtains non-negativity of the quantile density estimator. For larger $p$ the estimator can have negative values, similarly as kernel density estimators with higher order kernels. 
\end{remark}

\begin{remark}\label{Lemma: Schätzer gleich für p und p+1}\rm For interior design points $u=\frac{k}{n}\in[h,1-h]$ (for $k\in\{1,\dots,n\}$) one can show identity of the estimators $\hat  Q_p^{(v)}(u)=\hat Q_{p+1}^{(v)}(u)$ if $p-v$ is even. In particular for the quantile density estimation ($v=1$) it holds that $\hat q_p(u)=\hat q_{p+1}(u)$ for odd $p$. 
\end{remark}

In the following we consider estimation in interior points and boundary points. With interior points it is meant that  $u\in (0,1)$ is fixed. Then for $n$ large enough one obtains  $u\in [h,1-h]$. With  boundary points we mean sequences $u=u_n=ch\in [0,h)$ or $u=u_n=1-ch\in(1-h,1]$ for $c\in [0,1)$. The asymptotic results in the following theorem can also be derived for interior sequences $u=u_n\in [h,1-h]$ converging to some value in $(0,1)$ under some adapted assumptions. 
 
To define the quantile function also in 0 and 1 we use $Q(u)=F^{-1}(u)= \inf\{x\in \mathbb{R}; \  F(x)\geq u\}$ for $u\in (0,1]$, and $Q(0)=Q(0+)$, $q(0)=Q'(0+)$, $q(1)=Q'(1-)$. Here and in the following we use the notation $g(u+)=\lim_{t\searrow u}g(t)$ for right limits and $g(u-)=\lim_{t\nearrow u}g(t)$ for left limits.

For the asymptotic results we need the following assumptions. 

\begin{ann}\rm
    \label{Annahme an den Kern}
    Let $K$ be a bounded, symmetric Lipschitz-continuous density with support $[-1,1]$. 
\end{ann}

\begin{ann}\rm
    \label{Annahme an die Quantilsfunktion}
    For some $\delta>0$ let the quantile function $Q$ be $(p+1)$-times continuously differentiable in $[u-\delta,u+\delta]$ in the interior case, and in $[0,\delta]$ in the lower boundary case, and $[1-\delta,1] $ in the upper boundary case. Let $q(u)>0$ in the interior case, $q(0+)>0$ in the lower boundary case and $q(1-)>0$ in the upper boundary case. 
\end{ann}

 Note that for the boundary cases we assume that $Q^{(v)}(0+)\in \mathbb{R}$ and $Q^{(v)}(1-)\in\mathbb{R}$, respectively, for $v\in\{0,\dots,p\}$, which in particular means that the lower or upper bound of the support of $X$ is finite.   We consider the case of unbounded support in section 4. 

\medskip

\begin{ann}\rm
    \label{Annahme: Konvergenzen der Bandbreite}
    Assume that $h \rightarrow 0$ and $nh^{2p+1} = O(1)$ and 
   $       \displaystyle \frac{nh^2}{\log(n)^2\log\log(n)} \rightarrow \infty.$
\end{ann}

\medskip

 {
\begin{satz}
    \label{Satz: Konvergenz der Schätzer}
    Let assumptions \ref{Annahme an den Kern}, \ref{Annahme an die Quantilsfunktion} and \ref{Annahme: Konvergenzen der Bandbreite} hold.
  Define 
     \begin{eqnarray*}
                 \Gamma_{p} &=&  \int_I \int_I (x \wedge t)r_p(x)r_p(t)^\top K(x)K(t)\, dxdt \\
            c_{p} &=& \int_I r_p(x)x^{p+1}K(x)\, dx\\
             S_{p} &=& (\mu_{i+l})_{\substack{\scriptstyle i=0,\dots,p\\ \scriptstyle l = 0,\dots,p}} \mbox{ with } \mu_j = \int_I x^jK(x)\, dx.
    \end{eqnarray*}
Here the interval $I$ depends on the interior and boundary cases, 
$$I=\begin{cases} [-1,1] , & \text{fixed } u\in (0,1)\\ [-c,1], & u=ch, c\in [0,1)
\\ [-1,c], & u=1-ch, c\in [0,1).\end{cases}$$
Further define
$$ \mathcal{B}_{p,v}(u) = v!e_v^\top S_{p}^{-1}c_{p} \frac{Q^{(p+1)}(u)}{(p+1)!}$$
 for $v\in\{0,\dots,p\}$.
 Then for interior and boundary points $u$ it holds that
    \begin{eqnarray*}
        \frac{\sqrt{nh^{2v-1}}}{q(u)}\left(\hat{Q}_p^{(v)}(u)-Q^{(v)}(u)-h^{p+1-v}\mathcal{B}_{p,v}(u)\right) \stackrel{d}{\longrightarrow} \mathcal{N}(0,\mathcal{V}_{p,v})
    \end{eqnarray*}
    for $v\in\{1,\dots,p\}$ with asymptotic variance
    $$   \mathcal{V}_{p,v} = 
            (v!)^2e_v^\top S_{p}^{-1}\Gamma_{p}S_{p}^{-1}e_v.$$
For $v=0$ and interior  $u\in(0,1)$ it holds that 
    \begin{eqnarray*}
        \frac{\sqrt{n}}{q(u)(u-u^2)^{1/2}}\left(\hat{Q}^{(0)}_p(u)-Q(u)
        \right)  \stackrel{d}{\longrightarrow} \mathcal{N}(0,1).
    \end{eqnarray*}
    And for $v=0$ and boundary $u=ch$ or $u=1-ch$ it holds that
        \begin{eqnarray*}
        \sqrt{\frac{n}{h}}\frac{1}{q(u)}\left(\hat{Q}^{(0)}_p(u)-Q(u)-h^{p+1}\mathcal{B}_{p,0}(u)\right)  \stackrel{d}{\longrightarrow} \mathcal{N}(0,\mathcal{V}_{p,0})
    \end{eqnarray*}
    with asymptotic variance
    $$      \mathcal{V}_{p,0}=\left(e_0^\top S_{p}^{-1}\Gamma_{p}S_{p}^{-1}e_0 + c\right).$$
\end{satz}
}

\bigskip

The proof is given in the appendix.   In the  boundary cases $u=ch$ and $u=1-ch$, the terms  $q(u)$, $Q^{(p+1)}(u)$  in the results can be replaced by the right limits $q(0+)$, $Q^{(p+1)}(0+)$ in the lower case, and the left limits $q(1-)$, $Q^{(p+1)}(1-)$ in the upper case. The asymptotic variance does not depend on $u$ for $v\in\{1,\dots,p\}$.   {In the interior case for $v=0$  one obtains the same asymptotic result as for the empirical quantile function $Q_n=F_n^{-1}$,  see e.g.\ \cite{Serfling}, p.\ 77.} 

The estimators $\hat Q^{(v)}_p(u)$ are boundary adaptive as the classical local polynomial estimators by \cite{Fan-Gijbels}, and the bias rates and variance rates are also analogous to the density estimation result by \cite{cattaneo2020simple}, see Theorem 1 in their supplementary material. 
   
    From the proof it follows that the leading bias of $\hat{Q}_p^{(v)}(u)$ is given by
    \begin{eqnarray*}
        h^{p+1-v}\left(\frac{Q^{(p+1)}(u)}{(p+1)!}v!e_v^\top S_{p}^{-1}c_{p} + h\frac{Q^{(p+2)}(u)}{(p+2)!}v!e_v^\top S_{p}^{-1}\Tilde{c}_{p}\right)
    \end{eqnarray*}
    if $Q$ is $(p+2)$-times continuously differentiable in the area around $u$ (similar to the proof of Lemma \ref{Lemma: Lemma 2} in the appendix). Here the first term vanishes if $p-v$ is even and $u$ is interior, and 
    \begin{eqnarray*}
        \Tilde{c}_{p} = \int_I r_p(x)x^{p+2}K(x)\, dx.
    \end{eqnarray*}
    In that case the bandwidth assumption 3 can be relaxed. 
    For the interior case if $p-v$ is even the asymptotic bias using polynomial order $p$ is the same as when using polynomial order $p+1$ because then $e_v^\top S_{p+1}^{-1}c_{p+1} = e_v^\top S_{p}^{-1}\Tilde{c}_{p}$. Analogously this holds for the asymptotic variances. 
    This result is not surprising, since we know from Remark \ref{Lemma: Schätzer gleich für p und p+1}, that for even $p-v$ the estimators $Q_p^{(v)}(u)$ and $Q_{p+1}^{(v)}(u)$ are identical for the interior design points $u = k/n$. 
     {For instance for $v=1$ to estimate the quantile density $q$ in the interior case the bias has the same rate $h^2$ for $p=1$ and $p=2$, but at the boundary the bias rate is $h$ for $p=1$ and $h^2$ for $p=2$. Therefore $p=2$ has advantages over $p=1$. This corresponds to the recommendation by \cite{Fan-Gijbels} in their section 3.2.2 to use $p-v$ odd. However, in the local polynomial regression estimation the bias term in the case $p-v$ even is more complex than for $p-v$ odd, because it depends not only on the regression derivatives, but also on the density derivatives. This is simpler here because only derivatives of $Q$ appear in both cases.}
The following table summarizes the bias rates. 
\small
\begin{eqnarray*}
  \begin{array}{|c|c|c|}
\hline
\text{bias rate} & p - v \text{ odd} & p - v \text{ even} \\
\hline
u \text{ interior} & h^{p+1-v} & h^{p+2-v} \\
\hline
u \text{ boundary} & h^{p+1-v} & h^{p+1-v} \\
\hline
\end{array}  
\end{eqnarray*}
\normalsize
 {
For an optimal bandwidth one minimizes the dominating term of the mean squared error (MSE). The derivative $Q^{(p+1)}(u)$ appears in the optimal bandwidth formula and can be replaced by an estimator $\hat Q^{(p+1)}(u)$ for a plug-in method.
The following table shows the MSE rates for $v\geq 1$. } 
\small
\begin{eqnarray*}
  \begin{array}{|c|c|c|}
\hline
\text{MSE rate} & p - v \text{ odd} & p - v \text{ even} \\
\hline
u \text{ interior} & n^{-(2v-1)/(2p+1)} & n^{-(2v-1)/(2p+3)} \\
\hline
u \text{ boundary} &n^{-(2v-1)/(2p+1)} & n^{-(2v-1)/(2p+1)} \\
\hline
\end{array}  
\end{eqnarray*}
\normalsize
 Apart from the bias  the only unknown term in the asymptotics is $q(u)$, consistently estimated by $\hat q(u)$. For example, consider an interior $u$, even $p-v$, $v\in\{1,\dots,p\}$, and $nh^{2p+3}\to 0$, then the bias term is negligible, and an asymptotic $(1-\alpha)$-confidence interval for $Q^{(v)}(u)$ is given by
$$ \hat Q_{p}^{(v)}(u)\pm \frac{|\hat q(u)|\mathcal{V}_{p,v}^{1/2}}{\sqrt{nh^{2v-1}}} \Phi^{-1}\left(1-\textstyle{\frac{\alpha}{2}}\right),$$ 
where $\Phi$ denotes the standard normal cdf.
As another example to apply Theorem \ref{Satz: Konvergenz der Schätzer} we consider estimating the \cite{parzen1979nonparametric} hazard score function $s(u)=q'(u)/(q(u))^2$ by $\hat s(u)=\hat Q_{p}^{(2)}(u)/(\hat Q_p^{(1)}(u))^2$ and obtain in the interior case 
$$\sqrt{n h^3} \left(\hat s(u)-s(u)-h^{p-1}\frac{B_{p,2}(u)}{q^{2}(u)}\right)\stackrel{d}{\longrightarrow} \mathcal{N}(0,e_2^\top [p]!V_{p,u}[p]!e_2/(q(u))^2).$$
For example in the case $p=2$ the asymptotic variance is equal to 
$4(t_{0,0}\mu_2^2-2t_{0,2}\mu_2+t_{2,2})/(q(u)$ $(\mu_4-\mu_2^2))^2$ with the notation $\mu_j$ from Theorem \ref{Satz: Konvergenz der Schätzer} and 
$t_{i,j}=\int_{-1}^1\int_{-1}^1 (x\wedge t)x^it^j K(x)K(t)\, dx dt$. 

\section{Discussion of Quantile Density Estimation}
	\label{Abschnitt: Comparison to other estimators}

In this section we investigate $\hat q=\hat Q^{(1)}$ as estimator for the quantile density function $q$. First we consider the case of a bounded support of $X$.
We compare the local polynomial quantile density estimator with classical kernel-based estimators, in the interior and boundary case. This shows that the new estimator has much better boundary properties. We also consider the case of unbounded support, which as far as we know has not been considered for the classical quantile density estimators in the literature. 

\subsection{Bounded Support}\label{subsec-boundedsupport}

From Theorem \ref{Satz: Konvergenz der Schätzer} we obtain for the quantile density estimator $\hat q(u)=\hat q_p(u)$ of order $p$
 \begin{eqnarray*}
        \sqrt{nh}\left(\hat{q}(u)-q(u)-h^{p}\mathcal{B}_{p,1}(u)\right) \stackrel{d}{\longrightarrow} \mathcal{N}(0,q(u)^2\mathcal{V}_{p,1}),
    \end{eqnarray*}
    in the interior case with $ \mathcal{V}_{p,1}= 
            e_1^\top S_{p}^{-1}\Gamma_{p}S_{p}^{-1}e_1 $. The same holds in the boundary case when replacing $q(u)^2$ by $q(0+)^2$ in the lower case and $q(1-)^2$ in the upper case. 
    
   For  an interior point $u$,  $\mathcal{B}_{p,1}(u)=0$ if $p$ is odd.     
    In the cases $p=1$ and $p=2$ we obtain the same bias term $h^2q^{\prime\prime}(u)\mu_4/(6\mu_2)$, and the asymptotic variance has the form $q(u)^2\int_{-1}^1\int_{-1}^1 (x\wedge t)xt K(x)K(t)\, dx dt/\mu_2^2$ (with $\mathcal{V}_{1,1}=\mathcal{V}_{2,1}$). 
  
  \medskip
 
An alternative approach for estimating $q$ is the following estimator, see \cite{parzen1979nonparametric}, \cite{falk1986estimation}, and \cite{jones1992estimating}, 
\begin{eqnarray*}
    \Tilde{q}_1(u) = \int_0^1 Q_n(x)\frac{1}{h^2}k\left(\frac{u-x}{h}\right)dx = \sum_{i=1}^n X_{(i)} \left(K_h\left(u-\frac{i-1}{n}\right)-K_h\left(u-\frac{i}{n}\right)\right),
\end{eqnarray*}
where we use the notation $\frac{d}{dx}K(x) = k(x)$, and $Q_n=F_n^{-1}$ is the quantile function based on the empirical distribution function. 
Theorem 2 of \cite{falk1986estimation} shows pointwise asymptotic normality for interior points of this estimate. 
    Let $u\in (0,1)$ and suppose that $Q$ is twice differentiable near $u$ with bounded second derivative. Then if $k$ has support $[-1,1]$ with $\int_{-1}^1 k(x)dx = 0$, and $h\rightarrow 0$ while $nh^2\rightarrow \infty$ it holds that
    \begin{eqnarray*}
        \sqrt{nh} \left( \Tilde{q}_1(u) - Tq(u) \right) \stackrel{d}{\longrightarrow} \mathcal{N}\left(0,q(u)^2\int K(x)^2dx\right),
    \end{eqnarray*}
    where $Tq(u) = \int_0^1 Q(x)k((u-x)/h)/h^2dx $. 
Assume a symmetric kernel $K$ with $K(1)=0$, then $\int xk(x)dx = -1$ and $\int_{-1}^1 x^2k(x)dx=0$, from which  one obtains by Taylor's expansion for $n$ large enough,
\begin{eqnarray}\label{Taylor-Falk}
    Tq(u) 
    &=& \int_{(u-1)/h}^{u/h}Q(u-hx)h^{-1}k(x)\, dx \;=\; \int_{-1}^1  Q(u-hx)h^{-1}k(x)\, dx \\
    &=& q(u) - h^2\frac{q^{\prime \prime}(u)}{6}\int_{-1}^1 x^3k(x)\, dx + o(h^2).\nonumber
\end{eqnarray}
Therefore the asymptotic bias has the form 
\begin{eqnarray*}
    E[\Tilde{q}_1(u)-q(u)] \approx -\frac{h^2}{6}q''(u)\int_{-1}^1 x^3k(x)\, dx=\frac{h^2}{2}q''(u)\mu_2
\end{eqnarray*}
by partial integration. 
 {By employing higher order kernels, the order of the
leading bias can be reduced to a rate $h^s$ for some $s> 2$. 
In the case considered here the only terms that depend on $u$ are the factor $q''(u)$  in the bias and  the factor $q(u)^2$  in the asymptotic variance, which are the same as for the new estimator $\hat q(u)$. Thus to compare the terms one only has to consider the different factors not depending on $u$.} For example for the Epanechnikov kernel $\hat q(u)$ has a smaller bias term, but a larger asymptotic variance than $\tilde q_1(u)$. 


For the boundary case $u = ch$ we have to adapt the integral limits in (\ref{Taylor-Falk}) for the classical kernel estimator  $\tilde q_1$. One cannot approximate $q(u)$ by $Tq(u)$ anymore, as $\int_{-c}^1 k(x) \neq 0$ and $\int_{-c}^1 xk(x) \neq -1$. In particular the leading term of $Tq(u)-q(u)$ is (if we assume right-continuity and boundedness of $Q$ at zero)
\begin{eqnarray*}
    \left|Q(0)h^{-1} \int_{-c}^1 k(x)dx\right| \rightarrow \infty 
\end{eqnarray*}
(if $Q(0) \neq 0$). Therefore the estimator $\tilde q_1$ does not work in the boundary case, whereas the new local polynomial estimator $\hat q$ has very good boundary properties.

\medskip

Another estimator considered by \cite{jones1992estimating} is a plug-in estimator 
\begin{eqnarray*}
   {\tilde q}_2 (u) = \frac{1}{f_n(Q_n(u))}
\end{eqnarray*}
for  $q(u) = 1/f(Q(u))$ with a kernel density estimator  $f_n(x) = (nh)^{-1}\sum_{i=1}^n K((X_i-x)/h)$ for $f(x)$, and the empirical quantile function $Q_n$. 
For fixed $u\in (0,1)$  the asymptotic bias is
$$-\frac{h^2}{2}\frac{q(u)q''(u)-3(q'(u))^2}{q^3(u)}\mu_2,$$
and the asymptotic variance  $q^3(u)\int K^2(x)\, dx$.
Due to the different structures it is not easy to compare the bias and asymptotic variance with those of the new estimator $\hat q$.  \cite{jones1992estimating} compares $\tilde q_1(u)$ and $\tilde q_2(u)$ using a bandwidth for the kernel density estimator depending on $q(u)$, but not one of the estimators is generally better than the other. 

Note  in the expansion of $\tilde q_2(u)-q(u)$ the term $-(f_n(Q(u))-f(Q(u)))/{f(Q(u))^2}$ dominates, see \cite{soni2012nonparametric}.

Now we consider again the boundary cases. The kernel density estimator $ f_n$ evaluated in $Q(u)$ has boundary problems if $Q(u)$ is close to the boundary of the support of density $f$, say $[x_L,x_U]$. For comparison with the boundary case of the new estimator $\hat q$, 
\\
consider a Taylor approximation of $Q(u)$ for $u=ch$, i.e.
\begin{eqnarray*}
    Q(ch) = Q(0) + chq(0) + O(h^2) = x_L + \frac{c}{f(x_L)}h + O(h^2).
\end{eqnarray*}
This implies that the estimator $\tilde q_2$ has boundary problems if $u=ch$ and $c<f(x_L)$ because then $Q(u)$ is in the boundary of the kernel estimator.  However, if $c \geq f(x_L)$ we have $Q(u)$ basically in the interior of the kernel density estimator, as $Q(ch)\gtrsim x_L + h$. Therefore the bias will be of rate $ O(h^2)$.  

\cite{soni2012nonparametric} also consider a smoothed version of $\tilde q_2$, say $\tilde q_3$, and show asymptotic normality. But they show that all the three estimators $\tilde q_1$, $\tilde q_2$ and $\tilde q_3$ do not perform well at the boundaries.

\subsection{Unbounded Support} 

If $X$ has unbounded support (or the density $f$ is zero in a support boundary) it holds that $q(u)\to\infty$ for $u\to 0$ and/or $u\to 1$, and the same holds for other derivatives $Q^{(v)}(u)$.  In the following we present a consistency result for the unbounded case. To be precise we show under some assumptions, that
\begin{eqnarray}\label{unbounded-result}
    \frac{\hat Q_p^{(v)}(u)}{Q^{(v)}(u)}\stackrel{P}{\longrightarrow} 1,
\end{eqnarray}
where $u=u_n$ may depend on the sample size, and in particular the edge cases $u=u_n\to 0$ and $u=u_n\to 1$ for $n\to\infty$ are of interest. 
\begin{ann}
	\label{Annahme für Csörgo 1978 Theorem 3}
\rm	 
Assume that $F$ is twice differentiable on $(x_L,x_U)$, where $-\infty \leq x_L = \sup\{x:F(x)=0\} < \inf\{x:F(x)=1\} = x_U \leq \infty$ and $f \neq 0$ on $(x_L,x_U)$.
	Assume that for some $0<\gamma < \infty$ it holds that 
	\begin{eqnarray*}
		\sup_{0<t<1} t(1-t)\frac{|f'(Q(t))|}{f^2(Q(t))} \leq \gamma.
	\end{eqnarray*}
	Further assume, that if $f(x_L+)=0$ ($f(x_U-)=0$), then $f$ is nondecreasing (nonincreasing) on an interval to the right of $x_L$ (left of $x_U$).
\end{ann}
This assumption is according to apply Theorem 3 of \cite{csorgHo1978strong} in the proof of the following theorem, with slight modifications to the assumptions found in \cite{csorgHo1982general}. 
Further define 
\begin{eqnarray}\nonumber
    \bar{I}_n(u) &:=& \{i \in\{1,\dots,n\}: K_h(i/n-u)\neq 0\}\,,\,\,\,i_L^n(u) := \min \bar{I}_n(u)\,\,,\,\,\, i_U^n(u) := \max \bar{I}_n(u) \\
    I_n(u) &:=& \left[\frac{i_L^n(u)}{n},\frac{i_U^n(u)}{n}\right].\label{Inu}
\end{eqnarray}

\begin{ann}
    \label{Annahme für unbounded support}
\rm  Assume that $i_U^n(u) < n$ for $n$ large enough, and assume, that $Q$ is $(p+1)$-times continuously differentiable in a neighborhood of $u$, where we mean for the cases that $u \rightarrow 0$ or $u \rightarrow 1$ an interval $[0,\delta]$ and $[1-\delta,1]$,  respectively. Further assume, that as $nh^{2v}\rightarrow \infty$ it holds 
    \begin{eqnarray*}
        \frac{\sup_{t \in I_n(u)}|q(t)|}{Q^{(v)}(u)} = o(h^vn^{1/2})
    , \quad   \frac{h^{p+1}\sup_{t\in I_n(u)}|Q^{(p+1)}(t)|}{Q^{(v)}(u)} = o(h^v).
    \end{eqnarray*}
\end{ann}

\begin{satz}
	\label{Satz: Konsistenz im unbegrenzten Randfall, modifizierte Annahmen}
	Under assumption \ref{Annahme an den Kern}, \ref{Annahme für Csörgo 1978 Theorem 3} and \ref{Annahme für unbounded support}, the consistency
    (\ref{unbounded-result}) holds. 
 \end{satz}
The proof is given in the appendix. 
Note that the assumption $i_U^n(u) < n$ is needed for the case that $u \rightarrow 1$, as otherwise in the proof we have a summand with $Q_n(1)-Q(1)$ where for unbounded support we would have $Q(1) = \infty$. By assuming that $i_U^n(u) < n$ we can ignore the last summand due to $K_h(1-u) = 0$. But this assumption implies, that we cannot look at $u$ in the upper boundary case as in section \ref{Section: Definition of the estimators and main result}, meaning $u = 1-ch$ for $c\in [0,1)$, as then $i_U^n(u) = n$. But still sequences of $u$ that converge to one are allowed, if the convergence is not too fast.
\begin{bsp}
    \label{Beispiel für unbounded support}
    \rm 
     {One can easily check, that all the assumptions of Theorem \ref{Satz: Konsistenz im unbegrenzten Randfall, modifizierte Annahmen} are satisfied for standard normal distributed and for Exp(1)-distributed samples,} if we are interested in $v\geq 1$ and choose $h = n^{-\beta_v}$ with $0<\beta_v < \frac{1}{2v}$ and let $u_n = 1-h^\gamma$ for some $\gamma\in (0,1)$. Note that this implies
    $u_n>1-h$, but $u_n \rightarrow 1$.
\end{bsp}

\section{Simulation}
\label{Abschnitt: Simulation}
We did simulations using R (\cite{R}) for different cases and compared the mean squared error at fixed points and the mean integrated squared error. We oriented our simulations on the simulations in \cite{chesneau2016nonparametric}.
For comparison we use the three estimators $\Tilde q_1, \Tilde q_2, \Tilde q_3$ introduced in section \ref{Abschnitt: Comparison to other estimators}.
For all estimators we used the triangular kernel and we tested different bandwidths. The used distributions are the beta distribution $\text{Beta}(0.5,0.5)$ and the generalized lambda distribution (GLD) with different parameters. We mostly used $p=2$ and for some cases $p=1$ or $p=4$ for our local polynomial estimator.\\
In tables \ref{Tabelle: MSE versch Schätzer, GLD(0,7,7,7)} and \ref{Tabelle: MSE versch Schätzer, Beta(0.5,0.5)} we listed the mean squared errors of our simulations at fixed points for the four different estimators for a GLD$(0,7,7,7)$ and a $\text{Beta}(0.5,0.5)$ sample. We repeated each simulation $500$ times and give the average squared error. For the tables we used $p=2$ for our estimator and used the bandwidths $h\in \{0.15,0.19,0.25,0.35\}$ for the sample sizes $n=100,n=200$ and $n=500$.  For data-based bandwidth selection we further  used a plug-in estimator of the asymptotic MSE-optimal bandwidth 
\begin{eqnarray}
\label{Gleichung: MSE bandwidth}
    h_n(u) = \tau^{1/5}\left(\frac{3q(u)}{q^{\prime\prime}(u)\mu_4}\right)^{2/5}n^{-1/5} \,\,\,,\,\,\, \tau = \int_{-1}^1\int_{-1}^1 (x\wedge t)xt K(x)K(t)\, dx dt,
\end{eqnarray}
which is based on the bias and variance formulas from subsection \ref{subsec-boundedsupport}. Here we used our local polynomial estimator with $p=2, h = 0.15$ to estimate $q(u)$ and $p=3,h=0.25$ to estimate $q''(u)$.\\
While the results in the interior are quite similar for all estimators, we see a lower mean squared error at the boundary points for the new suggested local polynomial estimator. And while for the other estimators a bandwidth of $h=0.15$ seems to give the best result, for the local polynomial estimator a slightly larger bandwidth $h=0.25$ seems to be better. We also see, that the estimated bandwidth performs well. {For a better comparison we plotted the MSE values of the Beta$(0.5,0.5)$ distribution from the tables in figure \ref{fig:Plots_MSE_PI}.
\begin{figure}[t]
	\centering
	\includegraphics[width=5.25cm, height=4.75cm]{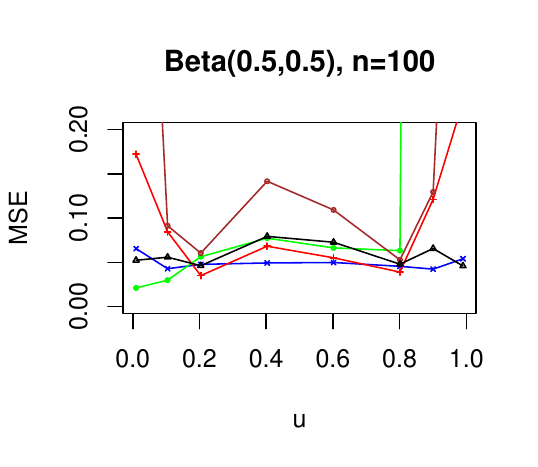}
	\includegraphics[width=5.25cm, height=4.75cm]{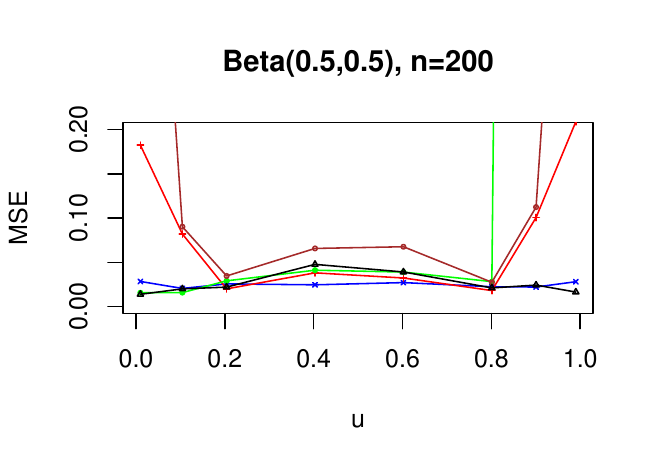}
	\includegraphics[width=5.25cm, height=4.75cm]{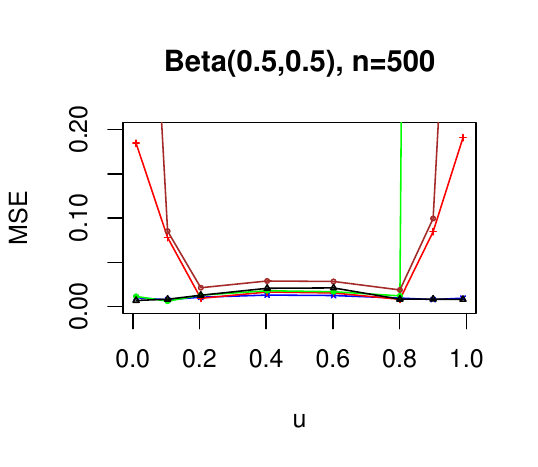}
	\captionsetup{font={footnotesize,stretch=1}}
	\caption{Plots of the MSE (mean squared error) after $500$ simulations using the Beta$(0.5,0.5)$ distribution. The blue crossed line is our local polynomial estimator, where we used $p=2$ and $h = 0.25$. The black line with the triangles is our local polynomial estimator with $p=2$ and $h(u)$ is the plug-in estimator from \eqref{Gleichung: MSE bandwidth}. $\Tilde q_1$ is the brown circled line, $\Tilde q_2$ is the green dotted line and $\Tilde q_3$ is the red line with the plus signs, for the other estimators we used $h=0.15$. The values are taken from table \ref{Tabelle: MSE versch Schätzer, Beta(0.5,0.5)}.}
	\label{fig:Plots_MSE_PI}
\end{figure}
Here we especially see the better convergence behaviour at the boundary and also more constant results in the interior, especially at smaller sample sizes for the fixed bandwidth of $h=0.25$ for our estimator.}  Our simulation results for the estimators $\Tilde q_1, \Tilde q_2, \Tilde q_3$ are  similar to the simulation results in \cite{chesneau2016nonparametric} in tables 2-4. And when comparing the results of the new local polynomial estimator we find better results at the boundary, also when compared to their wavelet-based quantile density estimator.
\begin{figure}[t]
	\centering
	\includegraphics[width=7.8cm, height=5cm]{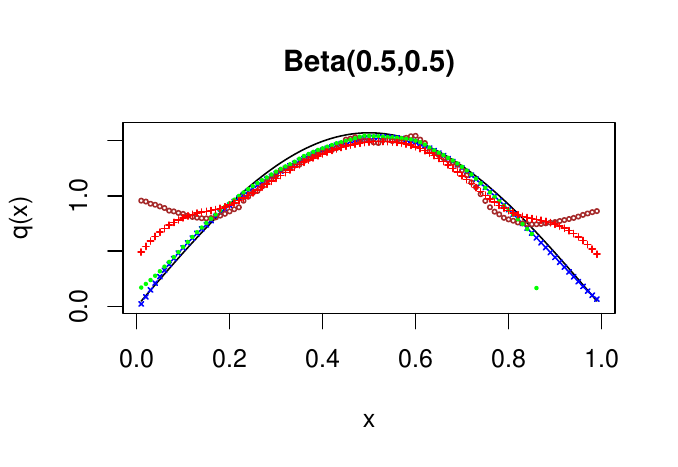}
	\includegraphics[width=7.8cm, height=5cm]{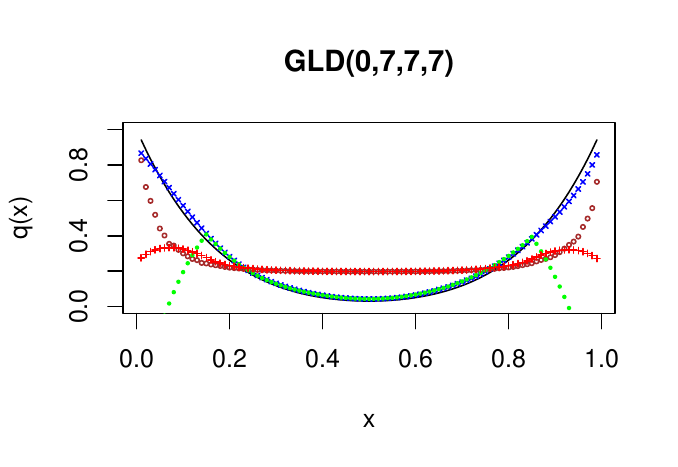}
    \captionsetup{font={footnotesize,stretch=1}}
	\caption{Plots of the average of $10$ simulations of sample size $n=200$ using the bounded Beta$(0.5,0.5)$, GLD$(0,7,7,7)$ distributions. The true $q$ is the black line. The blue crossed line is our local polynomial estimator, where we used $p=2$ and $h$ is the plug-in estimator of the global asymptotic IMSE optimal bandwidth from \eqref{Gleichung IMSE Bandwidth}. $\Tilde q_1$ is the brown circled line, $\Tilde q_2$ is the green dotted line and $\Tilde q_3$ is the red line  {with the plus signs}, for the other estimators we used $h=0.15$.}
    \label{fig:Plots_bounded_PI}
\end{figure}
In figures \ref{fig:Plots_bounded_PI} and \ref{fig:Plots_unbounded} we plotted the average of $10$ simulations for our estimator and compare it to the average of the estimators $\Tilde q_1,\Tilde q_2,\Tilde q_3$. We used the Beta$(0.5,0.5)$ and GLD$(0,7,7,7)$ distribution for two bounded cases and the $\mathcal{N}(0,1)$ and Exp$(1)$ distributions for two unbounded cases. In the bounded case we used $p=2$ for our estimator and a plug-in estimator for the asymptotic IMSE-optimal bandwidth 
\begin{eqnarray}
\label{Gleichung IMSE Bandwidth}
    h_n = \tau^{1/5}\left(\frac{3}{\mu_4}\right)^{2/5}\left(\frac{\int q(u)^2\, du}{\int (q^{\prime\prime}(u))^2\,du}\right)^{1/5}n^{-1/5},
\end{eqnarray}
where we estimate $q,q''$ as for the local bandwidth and approximate the integral on a grid.\\
For the unbounded cases we used $p=4$ and a fixed bandwidth of $h=0.19$, as we only have the consistency in the unbounded case and cannot use the MSE calculations based on Theorem \ref{Satz: Konvergenz der Schätzer}.
For the other estimators we used $h=0.15$, which is the one suggested and used by \cite{chesneau2016nonparametric} and \cite{soni2012nonparametric}, that also seemed to perform best with our data in table \ref{Tabelle: MSE versch Schätzer, GLD(0,7,7,7)} and \ref{Tabelle: MSE versch Schätzer, Beta(0.5,0.5)}. We clearly see the expected better results for the new suggested estimator at the boundaries in both the unbounded and bounded cases.\\
\begin{figure}
	\centering
	\includegraphics[width=7.8cm, height=5cm]{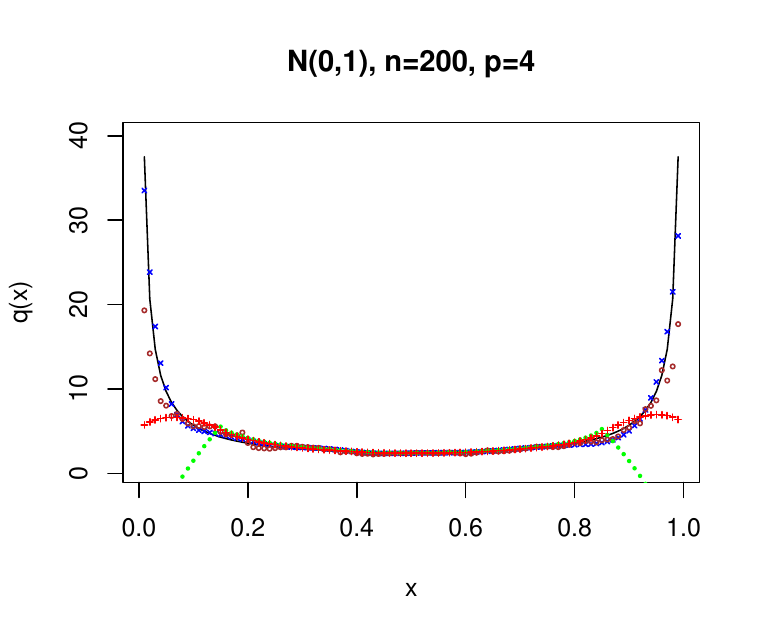}
	\includegraphics[width=7.8cm, height=5cm]{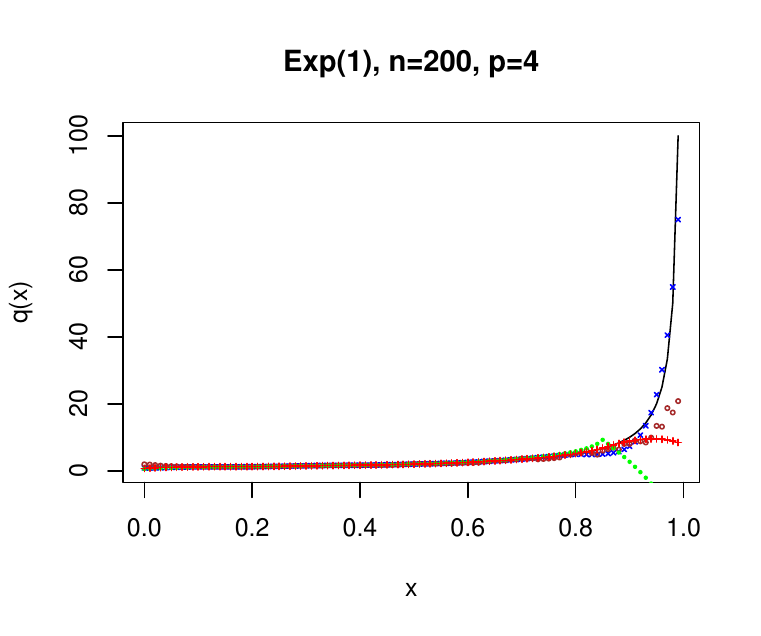}
    \captionsetup{font={footnotesize,stretch=1}}
	\caption{Plots of the average of $10$ simulations of sample size $n=200$ using the unbounded $\mathcal{N}(0,1)$, Exp$(1)$ distributions. We used $p=4,h=0.19$ for our estimator and $h=0.15$ for the other estimators. The lines are the same as in figure \ref{fig:Plots_bounded_PI}.}
    \label{fig:Plots_unbounded}
\end{figure}
For a global comparison we also looked at the mean integrated squared error, which is calculated as stated in the simulation section of \cite{chesneau2016nonparametric}. In table \ref{Tabelle: MISE-Values for different bandwidths} we compare the MISE of our local polynomial estimator for $p=2$ and for some cases also $p=1$ with different bandwidths to the MISE of the simulations in \cite{chesneau2016nonparametric}. We observe that for all sample sizes and distributions there is a combination of a bandwidth $h$ and $p$ such that the MISE of our estimator is smaller. This naturally brings up the question on how to choose the bandwidth and even the degree of the local polynomial smoothing. Especially for $n=500$ our estimator has a lower MISE for the $\text{Beta}(0.5,0.5)$, GLD$(0,7,7,7)$, GLD$(0.5,1,2,6)$ for all tested bandwidths except for $h=0.45$. We also calulated the MISE for the bandwidth from \eqref{Gleichung IMSE Bandwidth}. Here we also find better results for the $\text{Beta}(0.5,0.5)$, GLD$(0,7,7,7)$, GLD$(0.5,1,2,6)$ at $n=500$ compared to \cite{chesneau2016nonparametric}. \\
We also tested the accuracy of the confidence intervals constructed at the end of section \ref{Section: Definition of the estimators and main result}. We calculated the confidence interval for $1000$ samples of the GLD$(0.5,2,1.5,1.5)$ distribution and counted how often the true value  $q(u)$ was inside this interval. We tested this for $u=0,0.1,0.2,...,0.9,1$ and sample sizes $n=100,500,1000$ and used $p=2, h= 0.12,0.1,0.08$. In figure \ref{fig:Plots_KI} we see, that the coverage of the confidence intervals got closer to $1-\alpha$ with increasing sample size. Especially for the interior points we see very good results.
Also in figure \ref{fig:Plots_KI} we give the estimated density of $1000$ normalized estimations of $q(0.5)$ for the same setting that was used for the confidence intervals and compare it to the density of the standard normal distribution, which is the expected asymptotic distribution. We observe, that the standard normal distribution is a good estimation of this distribution and that we get closer to the standard normal distribution with increasing sample size. This also fits well to the good coverage rate of the confidence intervals, that were constructed using the normal distribution.

\begin{figure}[t]
	\centering
    \includegraphics[width=7.8cm, height=5cm]{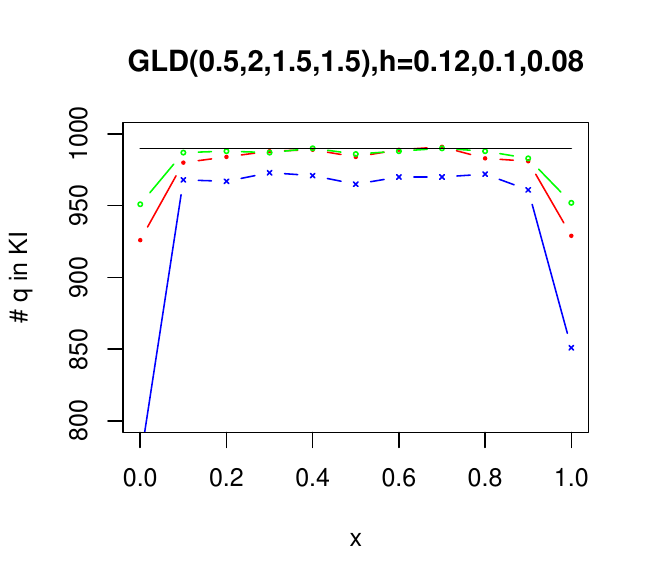}
	\includegraphics[width=7.8cm, height=5cm]{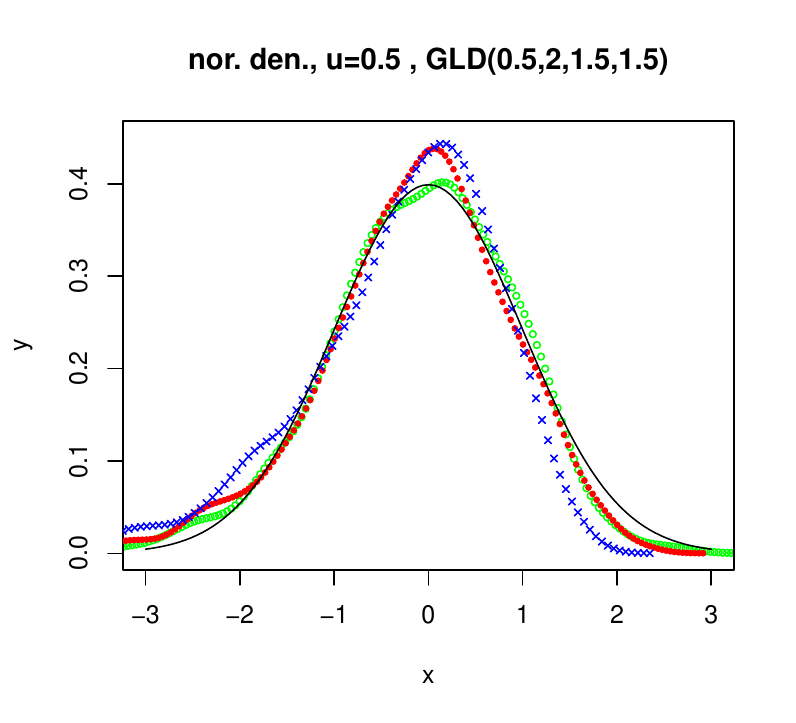}
    \captionsetup{font={footnotesize,stretch=1}}
	\caption{Left is a plot of the number of times, that $q(u)$ is inside the confidence interval for $\alpha = 0.01$ after $1000$ simulations. We used $u= 0,0.1,\dots,0.9,1$ for the GLD$(0.5,2,1.5,1.5)$ distribution and sample sizes of $n=100$ (blue  {crossed)}, $n=500$ (red  {dotted}) and $n=1000$ (green  {circled}). The black line is the expected $990$ observations in the confidence interval.\\
    For the same setting we estimated the normalized density of $1000$ simulations of $\hat{q}_2(u)$ and compared it to the standard normal distribution. We used $u=0.5$ and the colors and distribution are the same as above.}
    \label{fig:Plots_KI}
\end{figure}

\begin{table}[ht]
	\centering
	\resizebox{\textwidth}{!}{
		\begin{tabular}{ll|rrrrr|rrrrr|rrrrr}
			\toprule
			Sample Size & & \multicolumn{5}{c|}{$n=100$} & \multicolumn{5}{c|}{$n=200$} & \multicolumn{5}{c}{$n=500$} \\
			$u$ & $h$ & 0.15 & 0.19 & 0.25 & 0.35 & \text{PI} & 0.15 & 0.19 & 0.25 & 0.35 & \text{PI} & 0.15 & 0.19 & 0.25 & 0.35 & \text{PI}\\ 
			\toprule
			0.01 & loc & 0.1478 & 0.1168 & 0.0683 & 0.0590 & 0.1073 & 0.0683 & 0.0531 & 0.0347 & 0.0357 & 0.0592 & 0.0263 & 0.0216 & 0.0164 & 0.0257 & 0.0283\\ 
			& $\Tilde q_1$ & 0.0441 & 0.1231 & 0.1653 & 0.1413 & & 0.0324 & 0.1259 & 0.1662 & 0.1405 & & 0.0294 & 0.1245 & 0.1652 & 0.1412 &\\ 
			& $\Tilde q_2$ & 2.0033 & 1.6260 & 1.3309 & 1.1277 & & 2.0037 & 1.6373 & 1.3347 & 1.1275 & & 1.9984 & 1.6278 & 1.3367 & 1.1235 &\\ 
			&$\Tilde q_3$& 0.4692 & 0.5329 & 0.5413 & 0.4994 & & 0.4546 & 0.5216 & 0.5345 & 0.4944 & & 0.4492 & 0.5176 & 0.5295 & 0.4922 &\\
			[0.8em]
			0.1045 & loc & 0.0068 & 0.0062 & 0.0051 & 0.0052 & 0.0054 & 0.0037 & 0.0036 & 0.0034 & 0.0020 & 0.0024 & 0.0016 & 0.0017 & 0.0014 & 0.0008 & 0.0011 \\ 
			& $\Tilde q_1$ & 0.0526 & 0.0439 & 0.0276 & 0.0056 & & 0.0523 & 0.0416 & 0.0266 & 0.0053 & & 0.0538 & 0.0432 & 0.0268 & 0.0055 &\\ 
			& $\Tilde q_2$ & 0.1043 & 0.1698 & 0.2160 & 0.2430 & &  0.0973 & 0.1647 & 0.2133 & 0.2408 & & 0.0946 & 0.1610 & 0.2103 & 0.2397 &\\ 
			&$\Tilde q_3$& 0.0409 & 0.0500 & 0.0495 & 0.0363 & & 0.0405 & 0.0481 & 0.0483 & 0.0353 & & 0.0410 & 0.0483 & 0.0474 & 0.0351 &\\ 
			[0.8em]
			0.2040 & loc & 0.0049 & 0.0039 & 0.0035 & 0.0048 & 0.0044 & 0.0025 & 0.0027 & 0.0027 & 0.0042 & 0.0038 & 0.0011 & 0.0015 & 0.0024 & 0.0034 & 0.0028 \\ 
			& $\Tilde q_1$ & 0.0014 & 0.0002 & 0.0029 & 0.0227 & & 0.0014 & 0.0001 & 0.0029 & 0.0227 & & 0.0013 & 0.0001 & 0.0029 & 0.0225 &\\ 
			& $\Tilde q_2$ & 0.0038 & 0.0031 & 0.0052 & 0.0246 & & 0.0022 & 0.0025 & 0.0043 & 0.0241 & & 0.0011 & 0.0017 & 0.0037 & 0.0241 &\\ 
			&$\Tilde q_3$& 0.0008 & 0.0004 & 0.0043 & 0.0160 & & 0.0008 & 0.0003 & 0.0043 & 0.0162 & & 0.0008 & 0.0002 & 0.0043 & 0.0162 &\\ 
			[0.8em]
			0.4030 & loc & 0.0007 & 0.0008 & 0.0011 & 0.0028 & 0.0039 & 0.0003 & 0.0005 & 0.0007 & 0.0023 & 0.0018 & 0.0002 & 0.0002 & 0.0005 & 0.0019 & 0.0007 \\ 
			& $\Tilde q_1$ & 0.0220 & 0.0341 & 0.0585 & 0.1160 & & 0.0218 & 0.0343 & 0.0586 & 0.1159 &  &0.0217 & 0.0341 & 0.0587 & 0.1158& \\ 
			& $\Tilde q_2$ & 0.0006 & 0.0008 & 0.0012 & 0.0037 & & 0.0003 & 0.0005 & 0.0010 & 0.0034 & & 0.0002 & 0.0003 & 0.0008 & 0.0031 &\\ 
			&$\Tilde q_3$& 0.0222 & 0.0345 & 0.0595 & 0.1192 & & 0.0220 & 0.0347 & 0.0595 & 0.1191 & & 0.0219 & 0.0345 & 0.0596 & 0.1188 &\\ 
			[0.8em]
			0.6020 & loc & 0.0006 & 0.0006 & 0.0009 & 0.0023 & 0.0036 & 0.0003 & 0.0004 & 0.0007 & 0.0021 & 0.0018 & 0.0001 & 0.0003 & 0.0006 & 0.0020 & 0.0006 \\ 
			& $\Tilde q_1$ & 0.0214 & 0.0334 & 0.0577 & 0.1147 & & 0.0213 & 0.0336 & 0.0576 & 0.1146 & & 0.0211 & 0.0334 & 0.0577 & 0.1145 &\\ 
			& $\Tilde q_2$ & 0.0007 & 0.0008 & 0.0013 & 0.0037 & & 0.0004 & 0.0006 & 0.0010 & 0.0034 & & 0.0002 & 0.0004 & 0.0009 & 0.0033 &\\ 
			&$\Tilde q_3$& 0.0216 & 0.0337 & 0.0585 & 0.1176 & & 0.0214 & 0.0340 & 0.0585 & 0.1176 & & 0.0213 & 0.0338 & 0.0587 & 0.1175 &\\ 
			[0.8em]
			0.8010 & loc & 0.0040 & 0.0032 & 0.0024 & 0.0028 & 0.0029 & 0.0023 & 0.0020 & 0.0022 & 0.0029 & 0.0028 & 0.0011 & 0.0013 & 0.0021 & 0.0030 & 0.0025 \\ 
			& $\Tilde q_1$ & 0.0020 & 0.0003 & 0.0023 & 0.0205 & & 0.0020 & 0.0002 & 0.0021 & 0.0202 & & 0.0021 & 0.0001 & 0.0021 & 0.0200 &\\ 
			& $\Tilde q_2$ & 0.0037 & 0.0033 & 0.0069 & 0.0288 & & 0.0023 & 0.0024 & 0.0063 & 0.0283 & & 0.0012 & 0.0018 & 0.0056 & 0.0279 &\\ 
			&$\Tilde q_3$& 0.0015 & 0.0002 & 0.0032 & 0.0140 & & 0.0014 & 0.0001 & 0.0031 & 0.0138 & & 0.0013 & 0.0001 & 0.0032 & 0.0136 &\\ 
			[0.8em]
			0.9005 & loc & 0.0079 & 0.0055 & 0.0057 & 0.0057 & 0.0084 & 0.0033 & 0.0028 & 0.0026 & 0.0027 & 0.0033 & 0.0015 & 0.0012 & 0.0011 & 0.0009 & 0.0009 \\ 
			& $\Tilde q_1$ & 0.0551 & 0.0466 & 0.0304 & 0.0076 & & 0.0565 & 0.0468 & 0.0309 & 0.0077 & & 0.0585 & 0.0481 & 0.0310 & 0.0077 &\\ 
			& $\Tilde q_2$ & 0.1341 & 0.1964 & 0.2423 & 0.2660 & & 0.1252 & 0.1945 & 0.2389 & 0.2641 & & 0.1223 & 0.1917 & 0.2375 & 0.2631 &\\ 
			&$\Tilde q_3$& 0.0495 & 0.0575 & 0.0559 & 0.0418 & & 0.0481 & 0.0563 & 0.0563 & 0.0424 & & 0.0476 & 0.0566 & 0.0561 & 0.0428 &\\
			[0.8em]
			0.99 & loc & 0.1090 & 0.0834 & 0.0668 & 0.0626 & 0.1119 & 0.0596 & 0.0473 & 0.0372 & 0.0416 & 0.0065 & 0.0275 & 0.0205 & 0.0188 & 0.0295 & 0.0302\\ 
			& $\Tilde q_1$ & 0.0956 & 0.1716 & 0.1946 & 0.1559 & & 0.0619 & 0.1496 & 0.1826 & 0.1501 & & 0.0392 & 0.1349 & 0.1728 & 0.1457& \\ 
			& $\Tilde q_2$ & 2.0028 & 1.6213 & 1.3382 & 1.1260 & & 1.9994 & 1.6368 & 1.3353 & 1.1247 & & 1.9936 & 1.6318 & 1.3408 & 1.1254 &\\ 
			&$\Tilde q_3$& 0.4422 & 0.5091 & 0.5210 & 0.4832 & & 0.4420 & 0.5103 & 0.5246 & 0.4867 & & 0.4432 & 0.5135 & 0.5257 & 0.4888 &\\
			\toprule
		\end{tabular}
	}
	\caption{The MSE for different estimators. The local polynomial quantile estimator uses $p=2$. The MSE is built as an average of $500$ simulations of a GLD$(0,7,7,7)$ distributed sample for the sample sizes $n=100$ (left), $200$ (middle) and $500$ (right). In the \textit{PI}-column we used a plug-in estimator for the local asymptotic MSE-optimal bandwidth from \eqref{Gleichung: MSE bandwidth}.}
	\label{Tabelle: MSE versch Schätzer, GLD(0,7,7,7)}
\end{table}

\begin{table}[ht]
	\centering
	\resizebox{\textwidth}{!}{
		\begin{tabular}{ll|rrrrr|rrrrr|rrrrr}
			\toprule
			Sample Size & & \multicolumn{5}{c|}{$n=100$} & \multicolumn{5}{c|}{$n=200$} & \multicolumn{5}{c}{$n=500$} \\
			$u$ & $ h$ & 0.15 & 0.19 & 0.25 & 0.35 & \text{PI} & 0.15 & 0.19 & 0.25 & 0.35 & \text{PI} & 0.15 & 0.19 & 0.25 & 0.35 & \text{PI}\\ 
			\toprule
			0.01 & lok & 0.0452 & 0.0486 & 0.0655 & 0.0847 & 0.0523 & 0.0175 & 0.0195 & 0.0284 & 0.0353 & 0.0139 & 0.0060 & 0.0081 & 0.0100 & 0.0147 & 0.0069\\ 
			& $\Tilde q_1$ & 0.7883 & 1.0142 & 1.3072 & 1.8107 & & 0.7633 & 0.9299 & 1.2396 & 1.7166 & & 0.7324 & 0.9369 & 1.2284 & 1.6874 &\\ 
			& $\Tilde q_2$ & 0.0211 & 0.0294 & 0.0446 & 0.0765 & & 0.0156 & 0.0208 & 0.0367 & 0.0664 & & 0.0115 & 0.0197 & 0.0334 & 0.0625 &\\ 
			&$\Tilde q_3$& 0.1721 & 0.2160 & 0.2707 & 0.3640 & & 0.1823 & 0.2137 & 0.2723 & 0.3602 & & 0.1845 & 0.2231 & 0.2789 & 0.3641 &\\ 
			[0.8em]
			0.1045 & lok & 0.0344 & 0.0362 & 0.0428 & 0.0503 & 0.0558 & 0.0184 & 0.0165 & 0.0205 & 0.0232 & 0.0200 & 0.0069 & 0.0079 & 0.0077 & 0.0083 & 0.0083\\ 
			& $\Tilde q_1$ & 0.0913 & 0.1753 & 0.3105 & 0.5876 & & 0.0902 & 0.1589 & 0.2976 & 0.5588 & & 0.0852 & 0.1591 & 0.2953 & 0.5493 &\\ 
			& $\Tilde q_2$ & 0.0298 & 0.0271 & 0.0251 & 0.0222 & & 0.0158 & 0.0127 & 0.0120 & 0.0106 & & 0.0061 & 0.0058 & 0.0051 & 0.0057 &\\ 
			&$\Tilde q_3$& 0.0845 & 0.1164 & 0.1430 & 0.1839 & & 0.0820 & 0.1020 & 0.1349 & 0.1727 & & 0.0780 & 0.1041 & 0.1359 & 0.1733 &\\
			[0.8em]
			0.2040 & lok & 0.0615 & 0.0508 & 0.0477 & 0.0418 & 0.0462 & 0.0323 & 0.0283 & 0.0256 & 0.0203 & 0.0220 & 0.0141 & 0.0115 & 0.0106 & 0.0091 & 0.0128\\ 
			& $\Tilde q_1$ & 0.0604 & 0.0300 & 0.0237 & 0.0620 & & 0.0346 & 0.0174 & 0.0113 & 0.0469 & & 0.0213 & 0.0091 & 0.0055 & 0.0436 &\\ 
			& $\Tilde q_2$ & 0.0563 & 0.0453 & 0.0404 & 0.0268 & & 0.0291 & 0.0256 & 0.0222 & 0.0152 & & 0.0129 & 0.0107 & 0.0098 & 0.0088 &\\ 
			&$\Tilde q_3$& 0.0350 & 0.0266 & 0.0350 & 0.0476 & & 0.0200 & 0.0144 & 0.0241 & 0.0372 & & 0.0093 & 0.0068 & 0.0190 & 0.0358 &\\  
			[0.8em]
			0.4030 & lok & 0.0900 & 0.0623 & 0.0493 & 0.0399 & 0.0793 & 0.0465 & 0.0359 & 0.0246 & 0.0257 & 0.0477 & 0.0204 & 0.0154 & 0.0130 & 0.0175 & 0.0207 \\ 
			& $\Tilde q_1$ & 0.1416 & 0.0951 & 0.0875 & 0.0720 & & 0.0657 & 0.0553 & 0.0417 & 0.0472 & & 0.0289 & 0.0222 & 0.0193 & 0.0291 & \\ 
			& $\Tilde q_2$ & 0.0773 & 0.0533 & 0.0429 & 0.0411 & & 0.0411 & 0.0310 & 0.0229 & 0.0302 & & 0.0179 & 0.0140 & 0.0133 & 0.0236 & \\ 
			&$\Tilde q_3$& 0.0682 & 0.0560 & 0.0597 & 0.0562 & & 0.0381 & 0.0334 & 0.0387 & 0.0469 & & 0.0165 & 0.0192 & 0.0276 & 0.0399 & \\
			[0.8em]
			0.6020 & lok & 0.0735 & 0.0652 & 0.0498 & 0.0440 & 0.0727 & 0.0450 & 0.0361 & 0.0271 & 0.0270 & 0.0391 & 0.0189 & 0.0150 & 0.0125 & 0.0189 & 0.0211\\ 
			& $\Tilde q_1$ & 0.1092 & 0.1091 & 0.0902 & 0.0888 & & 0.0676 & 0.0552 & 0.0482 & 0.0530 & & 0.0284 & 0.0223 & 0.0181 & 0.0354 &\\ 
			& $\Tilde q_2$ & 0.0663 & 0.0577 & 0.0458 & 0.0496 & & 0.0390 & 0.0318 & 0.0260 & 0.0329 & & 0.0168 & 0.0137 & 0.0133 & 0.0257 &\\ 
			&$\Tilde q_3$& 0.0551 & 0.0543 & 0.0549 & 0.0604 & & 0.0323 & 0.0311 & 0.0379 & 0.0459 & & 0.0154 & 0.0173 & 0.0272 & 0.0398 & \\ 
			[0.8em]
			0.8010 & lok & 0.0724 & 0.0521 & 0.0455 & 0.0374 & 0.0479 & 0.0316 & 0.0267 & 0.0224 & 0.0170 & 0.0211 & 0.0129 & 0.0117 & 0.0096 & 0.0096 & 0.0083\\ 
			& $\Tilde q_1$ & 0.0528 & 0.0295 & 0.0237 & 0.0592 & & 0.0275 & 0.0137 & 0.0123 & 0.0538 & & 0.0188 & 0.0070 & 0.0061 & 0.0495 \\ 
			& $\Tilde q_2$ & 0.0632 & 0.0462 & 0.7692 & 1.7397 & & 0.0282 & 0.0243 & 0.7532 & 1.7011 & & 0.0121 & 0.0108 & 0.7510 & 1.7070 \\ 
			&$\Tilde q_3$& 0.0390 & 0.0296 & 0.0418 & 0.0464 & & 0.0181 & 0.0168 & 0.0289 & 0.0422 & & 0.0082 & 0.0086 & 0.0209 & 0.0377 \\
			[0.8em]
			0.9005 & lok & 0.0421 & 0.0362 & 0.0423 & 0.0480 & 0.0656 & 0.0199 & 0.0168 & 0.0221 & 0.0207 & 0.0243 & 0.0066 & 0.0073 & 0.0079 & 0.0084 & 0.0084\\ 
			& $\Tilde q_1$ & 0.1295 & 0.2068 & 0.3591 & 0.6083 & & 0.1124 & 0.1897 & 0.3420 & 0.5974 & & 0.0995 & 0.1857 & 0.3231 & 0.5831 &\\ 
			& $\Tilde q_2$ & 4.8961 & 6.1606 & 5.6445 & 3.9635 & & 4.9809 & 6.2463 & 5.6507 & 3.9465 & & 5.0659 & 6.2426 & 5.7043 & 3.9755 &\\ 
			&$\Tilde q_3$& 0.1212 & 0.1419 & 0.1705 & 0.1943 & & 0.1006 & 0.1226 & 0.1563 & 0.1882 & & 0.0846 & 0.1165 & 0.1438 & 0.1803 &\\
			[0.8em]
			0.99 & lok & 0.0390 & 0.0547 & 0.0540 & 0.0828 & 0.0459 & 0.0212 & 0.0187 & 0.0281 & 0.0386 & 0.0164 & 0.0062 & 0.0080 & 0.0095 & 0.0162 & 0.0081\\ 
			& $\Tilde q_1$ & 0.8080 & 1.0022 & 1.2785 & 1.7225 & & 0.7752 & 0.9397 & 1.2605 & 1.6995 & & 0.7228 & 0.9324 & 1.2107 & 1.6564 &\\ 
			& $\Tilde q_2$ & 37.1985 & 23.3841 & 13.3294 & 6.3679 & & 37.3122 & 23.5355 & 13.3515 & 6.3798 & & 37.4573 & 23.5459 & 13.4262 & 6.4153 &\\ 
			&$\Tilde q_3$& 0.2312 & 0.2630 & 0.3142 & 0.3874 & & 0.2089 & 0.2397 & 0.2985 & 0.3768 & & 0.1908 & 0.2320 & 0.2836 & 0.3658  &\\
			\toprule
		\end{tabular}
	}
	\caption{The MSE for different estimators. The local polynomial quantile estimator uses $p=2$. The MSE is built as an average of $500$ simulations of a Beta$(0.5,0.5)$ distributed sample for the sample sizes $n=100$ (left), $200$ (middle) and $500$ (right). In the \textit{PI}-column we used a plug-in estimator for the local asymptotic MSE-optimal bandwidth from \eqref{Gleichung: MSE bandwidth}.}
	\label{Tabelle: MSE versch Schätzer, Beta(0.5,0.5)}
\end{table}

\begin{table}[h]
	\centering
	\resizebox{\textwidth}{!}{
\begin{tabular}{ll|rrrrr}
	\toprule
	 & Dist & Beta(0.5,0.5) & GL(0,7,7,7) & GL(0.5,1,2,6) & GL(0.5,2,1.5,1.5) & GL(0,1.5,1.5,1.5)\\
	\midrule
	$n=200$ & $h$/Wav & 0.0252 & 0.0078 & 0.1149 & 0.0101 & 0.0189\\
	 \midrule
	$p=1$ & 0.15 & 0.0409 & \underline{0.0077} & 0.1482 & 0.0235 & 0.0387\\
	& 0.19 & 0.0413 & 0.0105 & 0.1635 & 0.0175 & 0.0303\\
	& 0.25 & 0.0499 & 0.0179 & 0.2408 & 0.0141 & 0.0240\\
	& 0.35 & 0.0818 & 0.0304 & 0.4019 & \underline{0.0097} & \underline{0.0175}\\
	& 0.45 & 0.1275 & 0.0443 & 0.5852 & \underline{0.0079} & \underline{0.0162}\\
    & \text{PI} & 0.0427 & 0.0081 & 0.1626 & 0.0198 & 0.0337 \\
	\addlinespace
	$p=2$ & 0.15 & 0.0344 & \underline{0.0057} & 0.1802 & 0.0283 & 0.0522\\
	& 0.19 & 0.0288 & \underline{0.0047} & 0.1406 & 0.0238 & 0.0435\\
	& 0.25 & \underline{0.0245} & \underline{0.0042} & 0.1175 & 0.0185 & 0.0355\\
	& 0.35 & \underline{0.0250} & \underline{0.0052} & \underline{0.1132} & 0.0156 & 0.0284 \\
	& 0.45 & 0.0344 & 0.0084 & 0.1282 & 0.0133 & 0.0248\\
    & \text{PI} & 0.0302 & \underline{0.0059} & 0.1752 & 0.0248 & 0.0441\\
	 \midrule
	  $n=500$ & $h$/Wav & 0.0167 & 0.0067 & 0.0866 & 0.0050 & 0.0093\\
	  \midrule
	$p=1$ & 0.15 & 0.0207 & \underline{0.0062} & 0.0974 & 0.0092 & 0.0173\\
	& 0.19 & 0.0245 & 0.0097 & 0.1352 & 0.0080 & 0.0136\\
	& 0.25 & 0.0374 & 0.0166 & 0.2134 & 0.0063 & 0.0106\\
	& 0.35 & 0.0729 & 0.0304 & 0.3997 & 0.0052 & \underline{0.0090}\\
	& 0.45 & 0.1234 & 0.0451 & 0.5858 & \underline{0.0050} & \underline{0.0092}\\
    & \text{PI} & 0.0240 & \underline{0.0057} & 0.1061 & 0.0781  & 0.0140 
    \\
	\addlinespace
	$p=2$ & 0.15 & \underline{0.0142} & \underline{0.0023} & \underline{0.0755} & 0.0120 & 0.0209\\
	& 0.19 & \underline{0.0113} & \underline{0.0021} & \underline{0.0630} & 0.0099 & 0.0161\\
	& 0.25 & \underline{0.0101} & \underline{0.0026} & \underline{0.0549} & 0.0081 & 0.0143\\
	& 0.35 & \underline{0.0137} & \underline{0.0042} & \underline{0.0695} & 0.0064 & 0.0114\\
	& 0.45 & 0.0220 & 0.0073 & 0.1024 & 0.0058 & 0.0101\\
    & \text{PI} & \underline{0.0120} & \underline{0.0028} & \underline{0.0795} & 0.0102 & 0.0179\\
	\bottomrule
\end{tabular}
}
\caption{The MISE values of $500$ simulations for different bandwidths, sample sizes $n=200$ and $n=500$ and $p=1,2$ compared to the MISE values of the wavelet estimator in the simulations of \cite{chesneau2016nonparametric} (underlined values are $\leq$ than the MISE of the wavelet estimator). In the \textit{PI}-column we used a plug-in estimator for the asymptotic IMSE-optimal bandwidth from \eqref{Gleichung IMSE Bandwidth}.}
\label{Tabelle: MISE-Values for different bandwidths}
\end{table}

\section{Concluding Remarks}\label{concluding remarks}

Quantile density estimation is important in different statistical areas. We suggested a new estimator that has better boundary properties than classical estimators, even in the case of unbounded data support. The estimator is very simple to apply as one can use local polynomial regression procedures for the (pseudo-)data $( F_n(X_i),X_i)$, $i=1,\dots,n$. The estimator can be applied in many areas where estimators for the quantile density function (sparsity function) are needed. Our general result gives joint asymptotic normality of estimators for the quantile function and the derivatives which then can be applied for asymptotic results of functions based on several derivatives, e.g.\ the score function for the hazard quantile function. 
Currently we are working on  rates of uniform convergence and simultaneous confidence bands based on the new estimator. 
The estimation procedure can also be generalized, e.g.\ to estimate conditional quantile density functions as considered by \cite{Xiang}.


\begin{appendix}

\section{Proofs}

 {
For the proof of Theorem \ref{Satz: Konvergenz der Schätzer} we consider the interior and boundary cases concurrently. 
Thus we use notations depending on $u$,
    \begin{eqnarray}\label{Vpvh}
        \mathcal{V}_{p,v,h}(u) = \begin{cases}
            (v!)^2q(u)^2e_v^\top S_{p,h,u}^{-1}\Gamma_{p,h,u}S_{p,h,u}^{-1}e_v, & 1\leq v\leq p\\
            q(u)^2(u-u^2) ,&   v=0, u\text{ interior} \\
            hq(u)^2\left(e_0^\top S_{p,h,u}^{-1}\Gamma_{p,h,u}S_{p,h,u}^{-1}e_0 + c\right), & v = 0 , \, u = ch \text{ or } u = 1-ch 
        \end{cases}
    \end{eqnarray}   
with
   \begin{eqnarray}
   \label{eq: Definition von Gamma_{p,h,u}}
                 \Gamma_{p,h,u} &=& \int_{-u/h}^{(1-u)/h} \int_{-u/h}^{(1-u)/h} (x \wedge t)r_p(x)r_p(t)^\top K(x)K(t)\, dxdt \nonumber\\ 
         S_{p,h,u} &=& (\mu_{i+l,h,u})_{\substack{\scriptstyle i=0,\dots,p\\ \scriptstyle l = 0,\dots,p}}  \mbox{ with }   \mu_{j,h,u} = \int_{-u/h}^{(1-u)/h} x^jK(x)\, dx. \label{eq: Definition von S_{p,h,u}}
    \end{eqnarray}
    Further define
     \begin{eqnarray}
   \label{eq: Definition von c_{p,h,u}}
                           c_{p,h,u} &=& \int_{-u/h}^{(1-u)/h} r_p(x)x^{p+1}K(x)\, dx .
    \end{eqnarray}
      Because $K$ has support $[-1,1]$ for the integrals one has to consider the intersection $[-1,1]\cap [-u/h,(1-u)/h] $, which converges for $h\to 0$ to the interval $I$ defined in Theorem \ref{Satz: Konvergenz der Schätzer} in the interior and boundary cases. Thus $\Gamma_{p,h,u}$,  $S_{p,h,u}$ and $c_{p,h,u}$ converge to $\Gamma_{p}$,  $S_{p}$ and $c_{p}$, respectively, defined in Theorem \ref{Satz: Konvergenz der Schätzer}, and by Slutsky's lemma we only need to prove that 
   \begin{eqnarray*}
    \label{Gleichung: Konvergenz Ableitungen von Q}
        \sqrt{\frac{nh^{2v-1}}{\mathcal{V}_{p,v,h}(u)}}\left(\hat{Q}_p^{(v)}(u)-Q^{(v)}(u)-h^{p+1-v}\mathcal{B}_{p,v}(u)\right) \stackrel{d}{\longrightarrow} \mathcal{N}(0,1)
    \end{eqnarray*}
    for $v\in\{1,\dots,p\}$, and
    \begin{eqnarray*}
        \label{Gleichung: Konvergenz Q}
        \sqrt{\frac{n}{\mathcal{V}_{p,0,h}(u)}}\left(\hat{Q}^{(0)}_p(u)-Q(u)-h^{p+1}\mathcal{B}_{p,0}(u)\right)  \stackrel{d}{\longrightarrow} \mathcal{N}(0,1)
    \end{eqnarray*}
    for $v=0$, where $\mathcal{V}_{p,v,h}$ is defined in (\ref{Vpvh}) and $\mathcal{B}_{p,v}(u)$ in Theorem \ref{Satz: Konvergenz der Schätzer}. Note that in the interior case $v=0$ the bias term can be ignored because of the bandwidth assumption \ref{Annahme: Konvergenzen der Bandbreite}. 
}

We use the notation $Q_n=F_n^{-1}$ for the empirical quantile function based on $X_1,\dots,X_n$ and from Remark \ref{matrizendarstellung} the matrix notations $D_{ {p,u}}$ and $K_{h,u}$. Further let  $D_{ {p},h,u} = [((i/n-u)/h)^j]_{1\leq i\leq n,0\leq j \leq p}$ and $H_{ {p}}$ the diagonal matrix with entries $1,h,\dots,h^p$.  {Note that those matrices also depend on the sample size $n$, but this is ignored in the notation.} Then it holds that $D^\top_{ {p},h,u} = H_{ {p}}^{-1}D_{ {p},u}^\top $,
and for $\beta_p(u)  = (Q^{(v)}(u)/v!)_{v=0,\dots,p}$, we obtain for the estimator defined in (\ref{Gleichung: Definition von beta über arg min mit i/n}) 
\begin{eqnarray}\label{hat-beta}
    \hat{\beta}_p(u)-\beta_p(u) = H_{ {p}}^{-1}\left(\frac{1}{n}D_{ {p},h,u}^\top K_{h,u}D_{ {p},h,u}\right)^{-1}\frac{1}{n}D_{ {p},h,u}^\top K_{h,u}\left(X_{(\cdot)}-D_{ {p},u}\beta_p(u)\right).
\end{eqnarray}
For the inverse matrix Lemma \ref{Lemma: Lemma 1} below gives us the asymptotic term. For  the numerator using $X_{(i)} = Q_n(i/n)$ we obtain the expansion 
\begin{eqnarray*}
    \frac{1}{n}D_{ {p},h,u}^\top K_{h,u}(X_{ {(\cdot)}}-D_{ {p,u}}\beta_p(u)) &=& \frac{1}{n}\sum_{i=1}^n r_p\left(\frac{i/n-u}{h}\right)\left(X_{(i)}-r_p(i/n-u)^\top \beta_p(u)\right)K_h(i/n-u) \nonumber\\
    &=&  \hat{A}_n(u) + \hat B_n(u) + \hat C_n(u)
\end{eqnarray*}    
    with the terms
 \begin{eqnarray}   
 \hat A_n(u) &=&   \frac{1}{n}\sum_{i=1}^n  r_p\left(\frac{i/n-u}{h}\right)(Q_n(i/n)-Q(i/n))K_h(i/n-u)\nonumber \\
    &&{}- \int_0^1 r_p\left(\frac{t-u}{h}\right)(Q_n(t)-Q(t))K_h(t-u)dt\label{termA}\\
   \hat B_n(u) &=& \frac{1}{n}\sum_{i=1}^n r_p\left(\frac{i/n-u}{h}\right)\left(Q(i/n)-r_p(i/n-u)^\top \beta_p(u)\right)K_h(i/n-u)\label{termB}\\
   \hat C_n(u) &=& \int_0^1 r_p\left(\frac{t-u}{h}\right)(Q_n(t)-Q(t))K_h(t-u)dt\label{termC}
\end{eqnarray}
treated in Lemmas \ref{Lemma: Lemma 2}--\ref{Lemma: Lemma 4} below.

\begin{lemma}
    \label{Lemma: Lemma 1}
    Assume assumption \ref{Annahme an den Kern} and let further $h \rightarrow 0$ and  $nh \rightarrow \infty$. We have
    \begin{eqnarray*}
        \frac{1}{n}D_{ {p},h,u}^\top K_{h,u}D_{ {p},h,u} = S_{ {p,h,u}} + O\left(\frac{1}{nh}\right) 
    \end{eqnarray*}
    with  {$S_{p,h,u}$ defined as in \eqref{eq: Definition von S_{p,h,u}}.}
\end{lemma}

\begin{proof}
  A generic element of $\frac{1}{n}D_{ {p},h,u}^\top K_{h,u}D_{ {p},h,u} $ takes the form 
    \begin{eqnarray*}
        \frac{1}{n}\sum_{i=1}^n \left(\frac{\frac{i}{n}-u}{h}\right)^jK_h\left(\frac{i}{n}-u\right)
    \end{eqnarray*}
    for $0\leq j \leq 2p$. This is a Riemann-sum and because $ x^jK(x)$ is Lipschitz-continuous on $[-1,1]$
 it is equal to
    \begin{eqnarray*}
       \int_0^1 \left(\frac{t-u}{h}\right)^j\frac{1}{h}K\left(\frac{t-u}{h}\right)dt + O\left(\frac{1}{nh}\right)
        &=& \int_{-u/h}^{(1-u)/h} s^jK(s)ds +  O\left(\frac{1}{nh}\right) \\
        &=&  {\mu_{j,h,u}} + O\left(\frac{1}{nh}\right)
    \end{eqnarray*}
     {with $\mu_{j,h,u}$ defined in (\ref{eq: Definition von S_{p,h,u}}). }
 \end{proof}

In the next step we consider $\hat B_n(u)$ from (\ref{termB}), which will turn out to be the leading bias term.

\begin{lemma}
\label{Lemma: Lemma 2}
    Let $u$ be either in the interior or in the boundary. Also let the assumptions \ref{Annahme an den Kern} and \ref{Annahme an die Quantilsfunktion} hold, and $h\rightarrow 0$, $nh \rightarrow \infty$. Then,  
    \begin{eqnarray*}
        \hat B_n(u) = 
            c_{ {p,h,u}} h^{p+1} \frac{Q^{(p+1)}(u)}{(p+1)!} +  o(h^{p+1}) 
    \end{eqnarray*}
with $ c_{ {p,h,u}}$ as in  {\eqref{eq: Definition von c_{p,h,u}}}. 
\end{lemma}

\begin{proof}
In the formula for $\hat B_n(u)$ from (\ref{termB}) due to the support of $K$ we only need to consider the summands where $i/n \in [u-h,u+h]$ (or $i/n \in [0,u+h]$ or $\in [u-h,1]$ in the respective boundary cases). By a Taylor-expansion up to order $p+1$ we get for the relevant $i$:
         \begin{eqnarray*}
        Q(i/n) &=&
        \sum_{j=0}^p (i/n-u)^j\frac{Q^{(j)}(u)}{j!} + (i/n-u)^{p+1}\frac{Q^{(p+1)}(u)}{(p+1)!}
         + (i/n-u)^{p+1} R_{p+1}^n(i,u),
    \end{eqnarray*}
    where the remainder term can be upper bounded  
    $$\max_{i}|R_{p+1}^n(i,u)|\leq \Tilde R^n(u)=\frac{\sup_{t: |t-u|\leq h} |Q^{(p+1)}(t)-Q^{(p+1)}(u)|}{(p+1)!}. $$
    Note that $r_p(i/n-u)^\top \beta_p(u) = \sum_{j=0}^p (i/n-u)^jQ^{(j)}(u)/j!$ and thus the $k$-th entry of $\hat B_n(u)$ becomes 
    \begin{eqnarray*}
        && \frac{1}{n}\sum_{i=1}^n \left(\frac{i/n-u}{h}\right)^k\left(((i/n-u)^{p+1}\frac{Q^{(p+1)}(u)}{(p+1)!}+(i/n-u)^{p+1}R_{p+1}^n(i)\right)K_h(i/n-u) \\
        &=& \frac{Q^{(p+1)}(u)}{(p+1)!} \frac{1}{n}\sum_{i=1}^n \left(\frac{i/n-u}{h}\right)^k(i/n-u)^{p+1}K_h(i/n-u) \\
        &+& \frac{1}{n}\sum_{i=1}^n \left(\frac{i/n-u}{h}\right)^kK_h(i/n-u)(i/n-u)^{p+1}R_{p+1}^n(i,u)\\
        &=& \frac{Q^{(p+1)}(u)}{(p+1)!}h^{p+1} \left(\int_{-u/h}^{(1-u)/h}x^{p+1}x^kK(x)dx + O\left(\frac{1}{nh}\right)\right) + o(h^{p+1})
    \end{eqnarray*}
    by Riemann sum approximation and as we can bound the remainder term 
    \begin{eqnarray*}
        &&\left|\frac{1}{n}\sum_{i=1}^n \left(\frac{i/n-u}{h}\right)^kK_h(i/n-u)(i/n-u)^{p+1}R_{p+1}^n(i,u)\right|\\
        &\leq& h^{p+1}\Tilde R^n(u) \frac{1}{n}\sum_{i=1}^n \left(\frac{|i/n-u|}{h}\right)^kK_h(i/n-u) = o(h^{p+1})
    \end{eqnarray*}
    as the sum converges to a finite integral and $\Tilde R^n(u) \rightarrow 0$ due to the continuity of $Q^{(p+1)}$ in $u$. Therefore we get that
    \begin{eqnarray*}
        \hat B_n(u) &=& h^{p+1} \frac{Q^{(p+1)}(u)}{(p+1)!} \int_{-u/h}^{(1-u)/h} r_p(x)x^{p+1}K(x)dx + o(h^{p+1})
    \end{eqnarray*}
\end{proof}

The dominating term for the asymptotic distribution is $\hat C_n(u)$ from (\ref{termC}) considered in the next lemma. 

\begin{lemma}
    \label{Lemma: Lemma 3}
    Assume the assumptions \ref{Annahme an den Kern}, \ref{Annahme an die Quantilsfunktion} and \ref{Annahme: Konvergenzen der Bandbreite}. Define the $(p+1)\times (p+1)$ scaling matrix
    \begin{eqnarray*}
        N_{ {h,u}} = \begin{dcases}
            diag(1,h^{-1/2},h^{-1/2},\dots,h^{-1/2}), & u\,\, \text{interior}\\
            diag(h^{-1/2},h^{-1/2},\dots,h^{-1/2}), & u \,\, \text{boundary.} 
        \end{dcases}
    \end{eqnarray*}
      Then
    \begin{eqnarray*}
        \sqrt{n}N_{ {h,u}}S_{ {p,h,u}}^{-1}\hat C_n(u) \stackrel{d}{\rightarrow} \mathcal{N}_{p+1}(0,\Sigma_{p,u}),
    \end{eqnarray*}
    where
    \begin{eqnarray*}
        \Sigma_{p,u} = \begin{dcases}
            q(u)^2(u-u^2)e_0e_0^\top  + q(u)^2(I-e_0e_0^\top )S_{ {p}}^{-1}\Gamma_{ {p}}S_{ {p}}^{-1}(I-e_0e_0^\top ), & u\,\text{interior}\\
            q(0+)^2(ce_0e_0^\top +S_{ {p}}^{-1}\Gamma_{ {p}}S_{ {p}}^{-1}) ,& u \, \text{lower}\, (u = ch)\\
            q(1-)^2(ce_0e_0^\top +S_{ {p}}^{-1}\Gamma_{ {p}}S_{ {p}}^{-1}-(e_1e_0^\top +e_0e_1^\top )), & u \,\text{upper}\, (u = 1-ch)
        \end{dcases}
    \end{eqnarray*}
       with  {$S_{p}$ and $\Gamma_p$ from Theorem \ref{Satz: Konvergenz der Schätzer}}.
\end{lemma}

\begin{proof}
    The proof uses a combination of the proof of Theorem 2 in \cite{falk1986estimation} and Lemma 3 of \cite{cattaneo2020simple}. Remind the notation $F_n$ for the empirical cdf of $X_1,\dots,X_n$, and $Q_n=F_n^{-1}$ the corresponding quantile function. Now let $U_i=F(X_i)$, then $U_1,U_2,\dots$ are independent and uniformly distributed on $[0,1]$. Denote with $\bar{F}_n$ the empirical cdf of $U_1,\dots,U_n$, and $\bar Q_n=\bar F_n^{-1}$ the corresponding quantile function. Then we have that 
    \begin{eqnarray*}
        Q_n(t) = Q(\bar Q_n(t))\text{ for } t \in (0,1).
    \end{eqnarray*}
    With that we now obtain for $\hat C_n(u)$ from (\ref{termC}) that
    \begin{eqnarray}\nonumber
        \hat C_n(u) &=& \int_0^1 r_p\left(\frac{t-u}{h}\right)(Q(\bar Q_n(t))-Q(t))K_h(t-u)dt\\
        &=& \int_{-u/h}^{(1-u)/h}r_p(x)(Q(\bar Q_n(u+xh))-Q(u+xh))K(x)dx\nonumber\\
        &=& \int_{-u/h}^{(1-u)/h}r_p(x)K(x)q(u+xh)(\bar Q_n(u+xh)-u-xh)dx + R_n(u).\label{dominatingC}
    \end{eqnarray}
    The last step is a Taylor-expansion up to order two and due to the boundedness of $q'$ in a neighborhood of $u$ the remainder term can be upper bounded by 
    \begin{eqnarray*}
    |R_n(u)|&\leq&   d \int_{-u/h}^{(1-u)/h} K(x)(\bar{Q}_n(u+xh)-(u+xh))^2dx
    \end{eqnarray*}
for some constant $d$. For uniformly distributed random variables it holds that
    \begin{eqnarray}
    \label{Gleichung für Taylor (glm Konv von emp Quantilprozess}
        \sup_{\frac{1}{n+1}\leq s \leq \frac{n}{n+1}} |\bar{Q}_n(s)-s| = O_P(n^{-1/2}),
    \end{eqnarray}
    see section 5 of \cite{csorgHo1985asymptotic},
    based on \cite{OReilly}.
    If $u$ is in the interior we effectively integrate over $[-1,1]$ for $n$ large enough, due to the kernel. So for $n$ large enough, we have $|\bar{Q}_n(s)-s|=O_P(n^{-1/2})$ uniformly in $s\in [u-h,u+h] \subseteq [\frac{1}{n+1},\frac{n}{n+1}]$, which then directly implies $R_n(u)=O_P(1/n)$.
    For $u = ch$ we have to be a bit more careful, as for $n$ large enough we now effectively integrate over $[-c,1]$. This implies we have $|\bar{Q}_n(s)-s|$ for $s\in [0,u+h]$.
    In order to still apply \cite{csorgHo1985asymptotic} we have to split the integral. For this let $l(u,n)$ be the largest $x\in [-c,1]$, such that $u+xh \leq \frac{1}{n+1}$. As the function $u+xh$ is continuous and strictly increasing in $x$, we get $u+l(u,n)h = \frac{1}{n+1}$ and $u+xh \geq \frac{1}{n+1}$ for $x \geq l(u,n)$. Also, since $nh \rightarrow \infty$ implies, that $h \geq \frac{a}{n}$ for any constant $a>0$ and $n$ large enough, we get that $u+h = h(c+1) \geq \frac{1}{n}\geq \frac{1}{n+1}$ for $n$ large enough. Therefore we can split $[0,u+h] = [0,\frac{1}{n+1})\cup[\frac{1}{n+1},u+h] $. Which implies, that $l(u,n) \geq -c$. Therefore we get
    \begin{eqnarray*}
        && \int_{-u/h}^{(1-u)/h} K(x)(\bar{Q}_n(u+xh)-(u+xh))^2dx \\
        &=& \int_{-c}^{l(u,n)}K(x)(\bar{Q}_n(u+xh)-(u+xh))^2dx + \int_{l(u,n)}^1K(x)(\bar{Q}_n(u+xh)-(u+xh))^2dx.
    \end{eqnarray*}
    For the second integral, we can use the same arguments as in the interior case to obtain $O_P(n^{-1})$. For the first integral we use boundedness of $K$ and substitution ($t = u+xh$) to obtain     
    for some constant $\tilde d$ the upper bound 
    \begin{eqnarray*}
        \frac{\tilde d}{h}\int_0^{1/(n+1)} (\bar{Q}_n(t)-t)^2dt 
   &\leq& \frac{2\tilde d}{h}\int_0^{1/(n+1)} \left(\frac{1}{n+1}-t\right)^2dt + \Delta_n^2\frac{2\tilde d}{h}\int_0^{1/(n+1)}dt\\
  &=& O(\frac{1}{n^3h})+O_P(\frac{1}{n^2 h}) =O_P(\frac{1}{n}) ,
    \end{eqnarray*}
    where we have applied $\bar{Q}_n(t) = \bar{Q}_n(\frac{1}{n+1}) = U_{(1)}$ and the notation $\Delta_n = \bar{Q}_n\left(\frac{1}{n+1}\right)-\frac{1}{n+1} = O_P(n^{-1/2})$, where the last step is again due to \cite{csorgHo1985asymptotic}.  {Here $U_{(1)}<\dots<U_{(n)}$ (with probability one) are the order statistics of $U_1,\dots,U_n$.}    
         For $u$ in the upper boundary we can make the same arguments, while using that $\bar{Q}_n(t) = \bar{Q}_n(\frac{n}{n+1}) = U_{(n-1)}$ for all $t \in [\frac{n}{n+1},1)$. Thus we also obtain in the boundary cases that $R_n(u)=O_P(1/n)$. 
    Now we do a Bahadur-approximation for the dominating term in (\ref{dominatingC}), 
    \begin{eqnarray}
    \label{Gleichung: Lemma 3: Bahadur-Schritt 1}
       \hat C_n(u) &=& \int_{-u/h}^{(1-u)/h}r_p(x)K(x)q(u+xh)(u+xh-\bar{F}_n(u+xh))dx \\
        &&{}+ \int_{-u/h}^{(1-u)/h}r_p(x)K(x)q(u+xh)\big(\bar{F}_n^{-1}(u+xh) \nonumber\\
        &&\qquad\qquad{}-u-xh - (u+xh-\bar{F}_n(u+xh))\big)dx+  O_P(n^{-1}). \nonumber
    \end{eqnarray}
      Using  the boundedness of $r,K$ and of $q$ in a neighborhood of $u$, 
  we obtain the upper bound of the second term on the right hand side of (\ref{Gleichung: Lemma 3: Bahadur-Schritt 1})
    \begin{eqnarray}
    \sup_{0\leq t\leq 1}|\bar{F}_n^{-1}(t)-t - (t-\bar{F}_n(t))| O(1)
    = O_P\left(\frac{\log(n)^{1/2}\log\log(n)^{1/4}}{n^{3/4}}\right)\label{CsörgöRevesz}
    \end{eqnarray}
     by Theorem E of \cite{csorgHo1978strong}. We obtain the dominating term 
    \begin{eqnarray*}
        \hat C_n(u) = \tilde C_n(u)  + O_P\left(\frac{\log(n)^{1/2}\log\log(n)^{1/4}}{n^{3/4}}\right)
    \end{eqnarray*}
    with 
  \begin{eqnarray}\nonumber
        \tilde {C}_n(u)
        &=&  \int_{-u/h}^{(1-u)/h}r_p(x)K(x)q(u+xh)(u+xh-\bar{F}_n(u+xh))dx\\
        &=& \frac{1}{n}\sum_{i=1}^n \int_{-u/h}^{(1-u)/h}r_p(x)K(x)q(u+xh)(u+xh-\mathds{1}_{U_i\leq u+xh})dx, 
        \label{C_n-entwicklung}
    \end{eqnarray}
    where $\mathds{1}_A$ denotes the indicator of an event $A$. 
       This expression is now comparable to the one in the proof of Lemma 3 in \cite{cattaneo2020simple} with the additional factor $q(u+xh)$  and the special case of a uniform distribution. The remainder term  multiplied by $\sqrt{n}N_{ {h,u}}S_{ {p,h,u}}^{-1}$ converges to zero since the matrices give an additional $\sqrt{nh^{-1}}$ term. The dominating term $\tilde C_n(u)$ is a sum of independent centered random variables, and we will apply the triangular array version of the central limit theorem with the Lyapunov-condition to show asymptotic normality. For this we calculate the  covariance matrix first
        (the integral limits in the following calculations are always $-u/h$ and $(1-u)/h$ as long as not stated otherwise), 
    \begin{eqnarray*}
        && \label{Gleichung: Varianz in Lemma 3 Teil 1}
        Cov\left(\int r_p(x)K(x)q(u+xh)(u+xh-\mathds{1}_{U_i\leq u+xh})\,dx\right) \\
        &=& \iint r_p(x)r_p(y)^\top K(x)K(y)q(u+xh)q(u+yh)\\
        &&\qquad{}E\left[\left(\mathds{1}_{U_i\leq u+xh}-u-xh\right)\left(\mathds{1}_{U_i\leq u+yh}-u-yh\right)\right]dxdy\\
        &=& \iint r_p(x)r_p(y)^\top K(x)K(y)q(u+xh)q(u+yh)\\
        &&\qquad{} (u+(x\wedge y)h-u^2-uyh-uxh-xyh^2)dxdy.
    \end{eqnarray*}
 For the next step we apply Taylor-expansion  $q(u+xh) = q(u)+xhq'(u)+o(h)$ for interior $u$  with simple calculations using the definitions of $e_0$, $e_1$, $S_{ {p,h,u}}$ and $\Gamma_{ {p,h,u}}$ to obtain 
     \begin{eqnarray*}
        && Cov\left(\int r_p(x)K(x)q(u+xh)(u+xh-\mathds{1}_{U_i\leq u+xh})dx\right) \nonumber\\
        &=& q(u)^2(u-u^2)S_{ {p,h,u}}e_0e_0^\top S_{ {p,h,u}} + hq(u)^2\Gamma_{ {p,h,u}}
        -hq(u)^2uS_{ {p,h,u}}(e_1e_0^\top +e_0e_1^\top )S_{ {p,h,u}}.\nonumber 
    \end{eqnarray*}
    Next we look at $u$ in the lower boundary, so $u = ch$ for some $c \in  {[0,1)}$, to obtain 
        \begin{eqnarray*}
        && Cov\left(\int K(x)r_p(x)q(u+xh)(u+xh-\mathds{1}_{U_i\leq u+xh})dx\right)\\ &=& \iint K(x)K(y)r_p(x)r_p(y)^\top q((c+x)h)q((c+y)h)(c+(x\wedge y))h dxdy + o(h)  \\
        &=& hq(0+)^2(cS_{ {p,h,u}}e_0e_0^\top S_{ {p,h,u}}+\Gamma_{ {p,h,u}})+o(h).
    \end{eqnarray*}
    Similarly for $u$ in the upper boundary, so $u = 1-ch$, and $    q(u+xh) = q(u) - (c-x)hq'(u)+o(h)$  we obtain 
    \begin{eqnarray*}
        &&Cov\left(\int K(x)r_p(x)q(u+xh)(u+xh-\mathds{1}_{U_i\leq u+xh})dx\right)\\ &=& \iint K(x)K(y)r_p(x)r_p(y)^\top q(1-)^2(c+(x\wedge y))h dxdy \\
        &&{}- h\iint (x+y)K(x)K(y)r_p(x)r_p(y)^\top q(1-)^2dxdy + o(h) \\
        &=& hq(1-)^2(cS_{ {p,h,u}}e_0e_0^\top S_{ {p,h,u}}+\Gamma_{ {p,h,u}}- S_{ {p,h,u}}(e_1e_0^\top +e_0e_1^\top )S_{ {p,h,u}})+o(h).
    \end{eqnarray*}
    With those expressions,
    \begin{eqnarray*}
        Cov(\sqrt{n}N_{ {h,u}}S_{ {p,h,u}}^{-1}\tilde{C}_n(u)) &=& nN_{ {h,u}}S_{ {p,h,u}}^{-1}Cov(\tilde {C}_n(u))S_{ {p,h,u}}^{-1}N_{ {h,u}} 
    \end{eqnarray*}
     {and the convergence of $S_{p,h,u} \rightarrow S_p$ and $\Gamma_{p,h,u} \rightarrow \Gamma_p$ seen at the beginning of the proof section} one obtains with simple calculations  {and Slutsky's lemma} the  covariance matrix formula stated in Lemma \ref{Lemma: Lemma 3}. 

To show asymptotic normality by     Cramér-Wold it is enough, that the Lyapunov-condition for $\delta = 2$ holds for all $a\in \mathbb{R}^{p+1}$ for $a^\top \hat C_n(u)$. Let $\mathcal{A} = [-u/h,(1-u)/h]^4\subseteq \mathbb{R}^4$, then
    \begin{eqnarray*}
        &&\sum_{i=1}^n E\left[\left|\frac{1}{\sqrt{n}}a^\top N_{ {h,u}}S_{ {p,h,u}}^{-1}\int_{-u/h}^{(1-u)/h}r_p(x)K(x)(\mathds{1}_{U_i\leq u+xh}-u-xh)q(u+xh)dx\right|^4\right]\\
        &\leq& \frac{1}{n}\iiiint_{\mathcal{A}}\prod_{j=1}^4|a^\top N_{ {h,u}}S_{ {p,h,u}}^{-1}r_p(x_j)K(x_j)q(u+x_jh)|dx_1dx_2dx_3dx_4 = O\left(\frac{1}{nh^2}\right).
    \end{eqnarray*}
 From this and the covariance calculations above the asserted asymptotic normality follows. 
\end{proof}

Next we consider the  term $\hat A_n(u)$ from (\ref{termA}) which is negligible. 

\begin{lemma}
    \label{Lemma: Lemma 4}
    Under the assumptions \ref{Annahme an den Kern}, \ref{Annahme an die Quantilsfunktion}, and \ref{Annahme: Konvergenzen der Bandbreite} there is an expansion
    \begin{eqnarray*}
        \hat A_n(u) = \hat{A}_{ {1,n}}(u) + \hat{A}_{2,n}(u)
    \end{eqnarray*} 
    with
    $ E[\hat{A}_{1,n}(u)] = 0$, $Cov(\hat{A}_{1,n}(u)) = O(n^{-2}h^{-1})$, and $\hat{A}_{2,n}(u) = O_P\left(\frac{\log(n)^{1/2}\log\log(n)^{1/4}}{n^{3/4}}\right)$.
\end{lemma}

\begin{proof} Note that with  $\hat A_n(u)$ in (\ref{termA}) and $\hat C_n(u)$ in (\ref{termC}) one can write
$$\hat A_n(u)=\hat D_n(u)-\hat C_n(u)$$
 {where we give $\hat D_n(u)$ below} and we start with a Taylor expansion of the first term as in Lemma \ref{Lemma: Lemma 3} using $Q_n(u)=Q(\bar Q_n(u))$,
  \begin{eqnarray*}
     \hat D_n(u)   &=&\frac{1}{n}\sum_{i=1}^nr_p\left(\frac{i/n-u}{h}\right)(Q_n(i/n)-Q(i/n))K_h(i/n-u) \\
        &=& \frac{1}{n}\sum_{i=1}^nr_p\left(\frac{i/n-u}{h}\right)q(i/n)(\bar{F}_n^{-1}(i/n)-\frac{i}{n})K_h(i/n-u) + O_P(n^{-1}).
    \end{eqnarray*}
The rate of the remainder term follows from boundedness of $q'$  in a neighborhood of $u$ and the upper bound
    \begin{eqnarray*}
        && d \sup_{i=1,\dots,n}|\bar{F}_n^{-1}(\frac{i}{n})-\frac{i}{n}|^2 \frac{1}{n}\sum_{i=1}^n K_h(i/n-u) \\
        &\leq& d\left(\sup_{1\leq i \leq n-1} |\bar{F}_n^{-1}(\frac{i}{n})-\frac{i}{n}|^2 + |\bar{F}_n^{-1}(1)-1|^2\right)\left(\int_{-u/h}^{(1-u)/h}K(u)du + O\left(\frac{1}{nh}\right)\right)  
    \end{eqnarray*}
    for some constant $d$. To obtain the rate $O_P(1/n)$ note that $\{\frac{1}{n},\dots,\frac{n-1}{n}\}\subseteq [\frac{1}{n+1},\frac{n}{n+1}]$ to  apply \eqref{Gleichung für Taylor (glm Konv von emp Quantilprozess} together with  
      $\bar{F}_n^{-1}(1)=U_{(n)}= 1+  O_P(n^{-1})$ as in  example 1.7.9 of \cite{leadbetter2012extremes}. 
        Then with (\ref{CsörgöRevesz}) by \cite{csorgHo1978strong} we obtain 
    \begin{eqnarray*}
        \hat D_n(u)&=& \frac{1}{n}\sum_{i=1}^n r_p\left(\frac{i/n-u}{h}\right)q(i/n)K_h(i/n-u)(\frac{i}{n}-\bar{F}_n(\frac{i}{n})) + O_P\left(\frac{\log(n)^{1/2} \log\log(n)^{1/4}}{n^{3/4}}\right) . 
\end{eqnarray*}  
      As seen in the proof of Lemma \ref{Lemma: Lemma 3} in (\ref{C_n-entwicklung}) we have 
    \begin{eqnarray*}
       \hat C_n(u) &=&  \frac{1}{n}\sum_{j=1}^n\int_0^1q(t)r_p\left(\frac{t-u}{h}\right)(t-\mathds{1}_{U_j\leq t})K_h(t-u)dt + O_P\left(\frac{\log(n)^{1/2} \log\log(n)^{1/4}}{n^{3/4}}\right) .
    \end{eqnarray*}
    To obtain the assertion of the lemma we define $\hat A_{ {2,n}}(u)$ as the sum of all $O_P$-remainder term considered above, and
  \begin{eqnarray*}
  \hat A_{ {1,n}}(u) &=& \frac{1}{n}\sum_{i=1}^n r_p\left(\frac{i/n-u}{h}\right)q(i/n)K_h(i/n-u)(\frac{i}{n}-\bar{F}_n(\frac{i}{n})) \\
  &&{}-\frac{1}{n}\sum_{j=1}^n\int_0^1q(t)r_p\left(\frac{t-u}{h}\right)(t-\mathds{1}_{U_j\leq t})K_h(t-u)dt
  \end{eqnarray*}   
    which is centered with covariance matrix 
    \begin{eqnarray*}
        &&Cov(\hat{A}_{1,n}(u)) = E[\hat{A}_{1,n}(u)(\hat{A}_{1,n}(u))^\top ]  \\
  &=& \frac{1}{n}\biggl( \frac{1}{n^2}\sum_{i=1}^n\sum_{j=1}^n\left(\frac{i\wedge 
        j}{n} - \frac{ij}{n^2}\right) q(i/n)q(j/n)r_p\left(\frac{i/n-u}{h}\right)r_p\left(\frac{j/n-u}{h}\right)^\top \\
        &&{}\qquad\qquad\times K_h(i/n-u)K_h(j/n-u) \\
        &&{}-2\int_0^1 \frac{1}{n}\sum_{i=1}^n q(i/n)q(t)r_p\left(\frac{i/n-u}{h}\right)r_p\left(\frac{t-u}{h}\right)K_h(i/n-u)K_h(t-u)\\
        &&{}\qquad\qquad\times \left((\frac{i}{n}\wedge t) - \frac{it}{n}\right)dt \\
        &&{}+ \iint_{[0,1]^2} q(t)q(s)K_h(t-u)K_h(s-u)r_p\left(\frac{t-u}{h}\right)r_p\left(\frac{s-u}{h}\right)((t\wedge s)-ts) dtds \biggr) \\
        &=& \frac{1}{n}\left(\Tilde{A}-2\Tilde{B}+\Tilde{C}\right).
    \end{eqnarray*}
    By Riemann sum approximations one obtains $\Tilde{A} = \Tilde{C} + O((nh)^{-1})$ and $\Tilde{B}= \Tilde{C} + O((nh)^{-1})$, and thus the remainder terms give the asserted rate  $Cov(\hat A_{ {1,n}}(u)) = O\left(\frac{1}{n^2h}\right)$.
\end{proof}

With these four lemmata we are now able to proof Theorem \ref{Satz: Konvergenz der Schätzer}.

\begin{proof}[Proof of Theorem \ref{Satz: Konvergenz der Schätzer}]

    We start with the case $1\leq v\leq p$.  By definition (\ref{Einführung: Gleichung: Defintion Schätzer}), Lemma  \ref{Lemma: Lemma 1} and (\ref{termA})--(\ref{termC}) we obtain 
    \begin{eqnarray*}
        && {\frac{\sqrt{nh^{2v-1}}}{q(u)}}\left(\hat{Q}_{ {p}}^{(v)}(u)-Q^{(v)}(u)-h^{p+1-v}\mathcal{B}_{p,v}(u)\right) \\
              &=&   {\frac{\sqrt{nh^{2v-1}}}{q(u)}}\Big(e_v^\top v!H_{ {p}}^{-1}(S_{ {p,h,u}}+ o(1))^{-1}\left(\hat{A}_{ {1,n}}(u) + \hat{A}_{ {2,n}}(u) + \hat B_n(u) + \hat C_n(u)\right) \\
              &&{}- h^{p+1-v}\mathcal{B}_{p,v}(u)\Big)\\
              &=&   {\frac{\sqrt{nh^{2v-1}}}{q(u)}}\left(e_v^\top v!H_{ {p}}^{-1}(S_{ {p,h,u}}+ o(1))^{-1}\left( \hat B_n(u) + \hat C_n(u)\right) - h^{p+1-v}\mathcal{B}_{p,v}(u)\right) +o_P(1).
    \end{eqnarray*}
     {In the last step we also applied}  $\sqrt{nh^{2v-1}}e_v^\top H_{ {p}}^{-1}=\sqrt{nh^{-1}} $, Lemma \ref{Lemma: Lemma 4} and assumption 3. 
By Lemma \ref{Lemma: Lemma 2} we get 
    \begin{eqnarray*}
         {\frac{\sqrt{nh^{2v-1}}}{q(u)}}\left(e_v^\top v!H_p^{-1}S_{ {p,h,u}}^{-1}\hat B_n(u) - h^{p+1-v}\mathcal{B}_{p,v}(u)\right) = 
        o\left(\sqrt{nh^{2p+1}}\right)
=o(1)    \end{eqnarray*}
    by assumption 3 and because $S_{ {p,h,u}}$ and $c_{ {p,h,u}}$ for $n\to\infty$ converge to $S_p$ and $c_p$, respectively.  {By Lemma \ref{Lemma: Lemma 3} one obtains the asymptotic normal distribution of $\hat C_n(u)$ and this yields that
        \begin{eqnarray*}
       {\frac{\sqrt{nh^{2v-1}}}{q(u)}}e_v^\top v!H^{-1}S_{ {p,h,u}}^{-1}\hat C_n(u) =  {\frac{1}{q(u)}}v!e_v^\top \sqrt{n}N_{h,u}S_{p,u}^{-1}\hat C_n(u) 
    \end{eqnarray*}
    converges to a centred normal distribution with asymptotic variance 
\begin{eqnarray*}
         \frac{(v!)^2}{q(u)^2}e_v^\top \Sigma_{p,u}e_v = (v!)^2e_v^\top S_{p}^{-1}\Gamma_{p}S_{p}^{-1}e_v = \mathcal{V}_{p,v}
    \end{eqnarray*}
     for $1 \leq v \leq p$ with $\mathcal{V}_{p,v}$ from Theorem \ref{Satz: Konvergenz der Schätzer} as the entry of $e_v^\top \Sigma_{p,u}e_v$ is not in the first column or row and therefore is given by the expression of $S_p^{-1}\Gamma_{p}S_p^{-1}$ for $v\neq 0$. Further note that $\sqrt{nh^{2v-1}}e_v^\top H_{ {p}}^{-1}=\sqrt{nh^{-1}} = \sqrt{n}e_v^{\top}N_{h,u}$.
    For $v=0$ and interior $u$ we get with the same arguments as above
    \begin{eqnarray*}
        &&\frac{\sqrt{n}}{q(u)(u-u^2)^{1/2}}\left(\hat{Q}_{ {p}}^{(0)}(u)-Q^{(0)}(u)        \right) \\
              &=&  \frac{\sqrt{n}}{q(u)(u-u^2)^{1/2}}\left(e_0^\top H_{ {p}}^{-1}S_{ {p,h,u}}^{-1}\hat C_n(u)\right) +o_P(1).
    \end{eqnarray*}
    The bias term can be ignored in this case because $\sqrt{n}h^{p+1}=o(1)$ by assumption \ref{Annahme: Konvergenzen der Bandbreite}.
    For interior $u$ we have that $\sqrt{n}e_0^\top H_{ {p}}^{-1} = \sqrt{n} = e_0^\top N_{h,u}$, so with the same argument as before we get with Lemma \ref{Lemma: Lemma 3} the asymptotic standard normality as 
    \begin{eqnarray*}
         \frac{1}{q(u)^2(u-u^2)}e_0^\top \Sigma_{p,u}e_0 = \frac{q(u)^2(u-u^2)}{q(u)^2(u-u^2)} =1.
    \end{eqnarray*}
    Similar for $v=0$ and upper boundary $u$ we have
    \begin{eqnarray*}
        &&\sqrt{\frac{n}{h}}\frac{1}{q(u)}\left(\hat{Q}_{ {p}}^{(0)}(u)-Q^{(0)}(u)-h^{p+1-v}\mathcal{B}_{p,0}(u)\right) \\
              &=& \sqrt{\frac{n}{h}}\frac{1}{q(u)}\left(e_0^\top H_{ {p}}^{-1}S_{ {p,h,u}}^{-1}\hat C_n(u)\right) +o_P(1)
    \end{eqnarray*}
    for boundary $u$ with $\sqrt{\frac{n}{h}}e_0^\top H_{ {p}}^{-1} = \sqrt{\frac{n}{h}} = \sqrt{n}e_0^\top N_{h,u}$. So again we get the asymptotic normality with Lemma \ref{Lemma: Lemma 3} and using that  $q(u) \rightarrow q(1-)$ to get with Slutsky's lemma the asymptotic variance  
    \begin{eqnarray*}
         \frac{1}{q(1-)^2}e_0^\top \Sigma_{p,u}e_0 &=& \frac{q(1-)^2}{q(1-)^2}e_0^\top \left(ce_0e_0^\top + S_p^{-1}\Gamma_pS_p^{-1} - (e_1e_0^\top + e_0e_1^\top) \right) e_0\\
         &=& c + e_0^\top S_p^{-1}\Gamma_pS_p^{-1}e_0 = \mathcal{V}_{p,0}.
    \end{eqnarray*}
    The lower boundary case works the same.
    }
\end{proof}
\begin{remark}
    \label{Bemerkung: Multivariate statement}
     {As we show a multivariate statement in Lemma \ref{Lemma: Lemma 3}, therefore with slight modifications in the proof we can also show asymptotic normality of the form
    \begin{eqnarray*}
		\frac{\sqrt{n}}{q(u)}N_n(\hat{\mathcal{Q}}_p(u)-\mathcal{Q}_p(u)-\mathcal{B}_n(u)) \stackrel{d}{\longrightarrow} \mathcal{N}_{p+1}(0,[p]!V_{p,u}[p]!),
	\end{eqnarray*}
    with $	\mathcal{Q}_p(u) = (Q^{(v)}(u))_{v=0,\dots,p}$, 
	$	\hat{\mathcal{Q}}_p(u) = (\hat Q_p^{(v)}(u))_{v=0,\dots,p}$. Here $N_n$, a matrix that collects the respective convergence rates and the covariance matrix $V_{p,u}$, differ in the interior and boundary case.
    }
\end{remark}

\medskip

Now we consider the unbounded case. 

\begin{proof}[Proof of Theorem \ref{Satz: Konsistenz im unbegrenzten Randfall, modifizierte Annahmen}]
    Using (\ref{Einführung: Gleichung: Defintion Schätzer}) and (\ref{hat-beta}) we need to show that 
    \begin{eqnarray*}
    \frac{\hat Q_p^{(v)}(u)-Q^{(v)}(u)}{Q^{(v)}(u)} &=&   e_v^\top v!H_{ {p}}^{-1}\left(\frac{1}{n}D_{ {p},h,u}^\top K_{h,u}D_{ {p},h,u}\right)^{-1}\frac{1}{nQ^{(v)}(u)}D_{ {p},h,u}^\top K_{h,u}\left(X_{(\cdot)}-D_{ {p},u}\beta_p(u)\right) \stackrel{P}{\longrightarrow}0.
    \end{eqnarray*}
    For this we consider the Euclidean norm and apply Lemma 
     \ref{Lemma: Lemma 1} to see that  $\|(\frac{1}{n}D_{ {p},h,u}^\top K_{h,u}D_{ {p},h,u})^{-1}\|$  {has the dominating term} $\|S_{p, {h},u}^{-1}\|$. One can show that $\sup_{u\in[0,1]}\|S_{p, {h},u}^{-1}\| $ is bounded  {independent of $h$}. Further, the term $e_v^\top H_{ {p}}^{-1}$ gives a factor of $h^{-v}$. 
 Now consider
 \begin{eqnarray*}
        \frac{1}{n}D_{ {p},h,u}^\top K_{h,u}(X_{(\cdot)}-D_{ {p},u}\beta_p(u)) &=& \frac{1}{n}\sum_{i=1}^n  r_p\left(\frac{i/n-u}{h}\right)(Q_n(i/n)-Q(i/n))K_h(i/n-u)\\
        &+& \frac{1}{n}\sum_{i=1}^n r_p\left(\frac{i/n-u}{h}\right)(Q(i/n)-r_p(i/n-u)^\top \beta_p(u))K_h(i/n-u)\\
        &=:& A_n(u) + B_n(u).
    \end{eqnarray*}
    By bounding $K$ by $C_K$ and using the support of $K$ we can upper bound
    \begin{eqnarray*}
		|A_n(u)| &\leq& C_K\left(2+\frac{1}{nh}\right)\sup_{t\in I_n(u)}|Q_n(t)-Q(t)|\\
		&\leq& 
        3C_K\sup_{t \in I_n(u)}|q(t)|n^{-1/2}\left(\sup_{t \in I_n(u)} |\rho_n(t)-u_n(t)| + \sup_{t\in I_n(u)} |u_n(t)|\right).
	\end{eqnarray*}
     {Here $I_n(u)$ is defined in (\ref{Inu}). Further} $\rho_n$ is the standardized quantile process and $u_n$ is the uniform quantile process,
\begin{eqnarray*}
    \rho_n(t) = n^{1/2}\frac{Q_n(t)-Q(t)}{q(t)}\,\,,\,\, u_n(t) = n^{1/2}( {\bar{F}_n(t)}-t),
\end{eqnarray*}
where  {$\bar{F}_n$} is the empirical quantile function based on the uniformly distributed $U_i = F(X_i)$, $i=1,\dots,n$.
    By  {assumption \ref{Annahme für unbounded support} we have $i_U^n(u) < n$, and} therefore $I_n(u) \subseteq [\frac{1}{n+1},\frac{n}{n+1}]$. Consequently by \cite{csorgHo1985asymptotic}, we have that $\sup_{t\in  {I_n(u)}}|u_n(t)| = O_P(1)$. Also by Theorem 3 of \cite{csorgHo1978strong} it holds that $\sup_{t\in  {I_n(u)}} |\rho_n(t)-u_n(t)| = O_P(1)$. With our  {assumption \ref{Annahme für unbounded support}} on the supremum of $q$ it follows that
    \begin{eqnarray*}
        \Big\vert\frac{A_n(u)}{Q^{(v)}(u)}\Big\vert &\leq& 3C_K \sup_{t \in I_n(u)}|q(t)| n^{-1/2} O_P(1) = 3C_K o(h^vn^{1/2})n^{-1/2}O_P(1) = o_P(h^v).
    \end{eqnarray*}
    Similarly we get with the typical Taylor-approximation argument that
    \begin{eqnarray*}
		\Big\vert\frac{B_n(u)}{Q^{(v)}(u)}\Big\vert &\leq& \frac{3C_K}{(p+1)!}\frac{h^{p+1}\sup_{t\in I_n(u)}|Q^{(p+1)}(t)|}{|Q^{(v)}(u)|} = o(h^v).
	\end{eqnarray*}
    Together we obtain
    \begin{eqnarray*}
        \frac{1}{Q^{(v)}(u)nh^v}D_{ {p},h,u}^\top K_{h,u}\left(X_{(\cdot)}-D_{ {p},u}\beta_p(u)\right) = o_P(1)
    \end{eqnarray*}
    which shows the consistency.
\end{proof}

\end{appendix}


\bibliography{bibliography-neu}
\bibliographystyle{apalike}


\end{document}